\documentclass[twoside,11pt]{article}

\usepackage{jmlr2e}
\usepackage{natbib}

\usepackage{amsmath,amssymb,mathrsfs}
\RequirePackage[colorlinks,citecolor=blue,urlcolor=blue]{hyperref}
\usepackage{arcs}
\usepackage[table,svgnames,x11names]{xcolor}
\usepackage{graphicx,wasysym,booktabs}
\usepackage[T1]{fontenc}
\usepackage{dsfont,courier}
\usepackage[scaled=.65]{helvet}

\def\a{\alpha}

\def\d{\delta}

\def\e{\epsilon}
\def\z{\zeta}
\def\t{\theta}
\def\T{\Theta}

\def\k{\kappa}

\def\s{\sigma}

\def\O{\Omega}
\def\ie{\textit{i.e., }}

\def\RR{\mathbb R}

\def\Sn{\mathfrak{S}_n}

\def\fcar{\mathds{1}}
\def\pinf{{+\infty}}

\def\suchthat{\,\big|\,}

\def\trace{\mathbf{tr}}
\def\ssi{\Longleftrightarrow}	

\def\esp{\mathbf E}

\def\prob{\mathbf P}

\def\trace{\mathbf{tr}}

\def\btheta{\boldsymbol{\theta}}

\def\bddiese{\mathbf{d}^\text{\tt\#}}
\def\bxdiese{\mathbf{x}^\text{\tt\#}}

\def\bX{\boldsymbol{X}}
\def\bXdiese{\boldsymbol{X}^{\text{\tt\#}}}

\def\Adiese{{A}^{\texttt{\#}}}
\def\bdiese{{b}^{\texttt{\#}}}

\def\Idiese{I^{\text{\tt\#}}}
\def\Xdiese{X^{\text{\tt\#}}}

\def\sdiese{\s^{\text{\tt\#}}}

\def\xidiese{\xi^\text{\tt\#}}

\def\bsigma{\boldsymbol{\sigma}}

\def\id{\textit{id}}
\def\tdiese{\t^\text{\tt\#}}
\def\bartdiese{\bar\t^\text{\tt\#}}
\def\barXdiese{\bar X^\text{\tt\#}}
\def\btdiese{\boldsymbol{\theta}^\text{\tt\#}}
\def\bsigmadiese{\boldsymbol{\sigma}^\text{\tt\#}}

\def\bQ{\mathbf Q}

\newcommand{\fonction}[5]{\begin{array}{l|rcl}
#1: & #2 & \longrightarrow & #3 \\
    & #4 & \longmapsto & #5 \end{array}}

\newenvironment{proofof}[1][Proof]{\begin{trivlist}
\item[\hskip \labelsep {\bfseries #1}]}{\end{trivlist}}

\jmlrheading{1}{2000}{1-48}{4/00}{10/00}{Olivier Collier and Arnak Dalalyan}

\ShortHeadings{Minimax permutation estimation}{Collier and Dalalyan}
\firstpageno{1}

\begin{document}

\title{Minimax rates in permutation estimation for feature matching}

\author{\name Olivier Collier \email olivier.collier@enpc.fr \\
       \addr Imagine- LIGM\\
       Universit\'e Paris EST\\
       Marne-la-Vall\'ee, FRANCE
       \AND
       \name Arnak S.\ Dalalyan \email arnak.dalalyan@ensae.fr \\
       \addr Laboratoire de Statistique\\
       ENSAE - CREST\\
       Malakoff, FRANCE}

\editor{Leslie Pack Kaelbling}

\maketitle

\begin{abstract}
The problem of matching two sets of features appears in various tasks of computer vision and
can be often formalized as a problem of permutation estimation. We address this problem from
a statistical point of view and provide a theoretical analysis of the accuracy of several
natural estimators. To this end, the minimax rate of separation is investigated and its expression
is obtained as a function of the sample size, noise level and dimension of the features.
We consider the cases of homoscedastic and heteroscedastic noise and establish, in each case,
tight upper bounds on the separation distance of several estimators. These upper bounds
are shown to be unimprovable both in the homoscedastic and heteroscedastic settings. Interestingly,
these bounds demonstrate that a phase transition occurs when the dimension $d$ of the features
is of the order of the logarithm of the number of features $n$. For $d=O(\log n)$, the rate is
dimension free and equals $\sigma (\log n)^{1/2}$, where $\sigma$ is the noise level. In contrast,
when $d$ is larger than $c\log n$ for some constant $c>0$, the minimax rate increases with $d$ and
is of the order of $\sigma(d\log n)^{1/4}$.
We also discuss the computational aspects of the estimators and provide empirical evidence of their
consistency on synthetic data. Finally, we show that our results extend to more general matching
criteria.
\end{abstract}


\section{Introduction}

\subsection{The Statistical Problem of Feature Matching}

In this paper, we present a rigorous statistical analysis of the problem of permutation estimation and multiple feature
matching from noisy observations. More precisely, let $\{X_1,\ldots,X_n\}$ and $\{\Xdiese_1,\ldots,\Xdiese_m\}$ be
two sets of vectors from $\RR^d$, hereafter referred to as noisy features, containing many matching elements. That is,
for many $X_i$'s there is a $\Xdiese_j$ such that $X_i$ and $\Xdiese_j$ coincide up to an observation noise (or measurement
error). Our goal is to estimate an application $\pi^* : \{1,\ldots,n\}\to \{1,\ldots,m\}$ for which each $X_i$ matches
with $\Xdiese_{\pi^*(i)}$ and to provide tight conditions which make it possible to accurately recover $\pi^*$ from data.


In order to define a statistical framework allowing us to compare different estimators of $\pi^*$, we
confine\footnote{These assumptions are imposed for the purpose of getting transparent theoretical results and are in no
way necessary for the validity of the considered estimation procedures, as discussed later in the paper.} our attention to
the case $n=m$, that is when the two sets of noisy features have equal sizes.  Furthermore, we assume that there exists a
unique permutation of $\{1,\ldots,n\}$, denoted $\pi^*$, leading to pairs of features $(X_i,\Xdiese_{\pi^*(i)})$ that match
up to a measurement error. In such a situation, it is clearly impossible to recover the true permutation $\pi^*$ if
some features within the set $\{X_1,\ldots,X_n\}$ are too close. Based on this observation, we propose to measure
the quality of a procedure of permutation estimation by the minimal distance between pairs of different features for which
the given procedure is still consistent. This quantity will be called \textit{separation distance} and will be the
main concept of interest in the present study. In this respect, the approach we adopted is close in spirit to
the minimax theory of hypotheses testing (see, for instance, \cite{Spokoiny96, Ingster}).

\subsection{A Motivating Example: Feature Matching in Computer Vision}

Many tasks of computer vision, such as object recognition, motion tracking or structure from motion, are currently
carried out using algorithms that contain a step of feature matching, cf.~\cite{Szeliski, Hartley_Zisserman}.
The features are usually local descriptors that serve to summarize the images. The most famous examples of such
features are perhaps SIFT \citep{Lowe2004} and SURF \citep{surf}. Once the features have been computed for each image,
an algorithm is applied to match features of one image to those of another one. The matching pairs are then used
for estimating the deformation of the object, for detecting the new position of the followed object, for creating a
panorama, etc. In this paper, we are interested in simultaneous matching of a large number of features. The main focus
is on the case when the two sets of features are extracted from the images that represent the same scene with a large
overlap, and therefore the sets of features are (nearly) of the same size and every feature in the first image is
also present in the second one. This problem is made more difficult by the presence of noise in the images, and
thus in the features as well. Typically, due to the high resolution of most images, the number of features is large
and their dimension is relatively large as well (128 for SIFT and 64 for SURF). It is therefore important to
characterize the behavior of various matching procedures as a function of the number of features, the dimension and
the noise level.

\subsection{Main Contributions}

We consider four procedures of permutation estimation that naturally arise in this context. The first one is a greedy
procedure that sequentially assigns to each feature $X_i$ the closest feature $\Xdiese_j$ among those features that have
not been assigned at an earlier step. The three other estimators are defined as minimizers of the (profiled-)log-likelihood under
three different modeling assumptions. These three modeling assumptions are that the
noise level is constant across all the features (homoscedastic noise), that the noise level is variable (heteroscedastic noise) but known
and that the noise level is variable and unknown. The corresponding estimators are respectively called least sum of squares (LSS) estimator,
least sum of normalized squares (LSNS) estimator and least sum of logarithms (LSL) estimator.

We first consider the homoscedastic setting and show that all the considered estimators are consistent under similar
conditions on the minimal distance between distinct features $\k$. These conditions state that $\k$ is larger than some function of
the noise level $\s$, the sample size $n$ and the dimension $d$. This function is the same for the four aforementioned procedures
and is given, up to a multiplicative factor, by
\begin{align}\label{kappastar}
\k^*(\s,n,d) = \s \max((\log n)^{1/2}, (d\log n)^{1/4}).
\end{align}
Then, we prove that this expression provides the optimal rate of the separation distance in the sense that for some
absolute constant $c$  if $\k \le c\k^*(\s,n,d)$ then there is no procedure capable of consistently estimating $\pi^*$.

In the heteroscedastic case, we provide an upper bound on the identifiability threshold ensuring the consistency of the LSNS and
LSL estimators. Up to a proper normalization by the noise level, this bound is of the same form as (\ref{kappastar}) and, therefore,
the ignorance of the noise level does not seriously affect the quality of estimation. Furthermore, the LSL estimator is easy
to adapt to the case $n\not = m$ and is robust to the presence of outliers in the features. We carried out a small experimental
evaluation that confirms that in the heteroscedastic setting the LSL estimator is as good as the LSNS (pseudo-) estimator and
that they outperform the two other estimators: the greedy estimator and the least sum of squares. We also argue that the three
estimators stemming from the maximum likelihood methodology are efficiently computable either by linear programming or by the
Hungarian algorithm.

Note that different loss functions may be used for measuring the distance between an estimated permutation and the true one. Most results
of this paper are established for the 0\,-1 loss, which equals one if the estimator and the true permutation differ at least at one location
and equals 0 otherwise. However, it is of interest to analyze the situation with the Hamming distance as well, since it amounts to controlling
the proportion of the mismatched features and, hence, offers a more graduated evaluation of the quality of estimation.  We show that in the
case of the Hamming distance, in the regime of moderately large dimension (\ie $d\ge c\log n$ for some constant $c>0$) the rate of separation
is exactly the same as in the case of the 0\,-1 distance. The picture is more complex in the regime of small dimensionality $d=o(\log(n))$,
in which we get the same upper bound as for the 0\,-1 loss but the lower bound is expressed in terms of the logarithm of the packing number
of an $\ell_2$ ball of the symmetric group. We conjecture that this quantity is of the order of $n\log(n)$ and check this conjecture for
relatively small values of $n$. If this conjecture is correct, our lower bound coincides up to a multiplicative factor with the upper bound.

Finally, let us mention that some of the results of the present work have been presented in the AI-STATS 2013
conference and published
in the proceedings \citep{jmlr_CD13}.

\subsection{Plan of the Paper}

We introduce in Section~\ref{section:model} a model for the problem of matching two sets of features and of estimation of a permutation.
The estimating procedures analyzed in this work are presented in Section~\ref{section:estimators}, whereas their performances in terms
of rates of separation distance are described in Section~\ref{section:performance}. In Section~\ref{section:computationalaspects}, computational aspects of the
estimating procedures are discussed while Section~\ref{section:extensions} is devoted to the
statement of some extensions of our results. We report in Section~\ref{section:simulations} the results of some numerical experiments.
The proofs of the theorems and of the lemmas are postponed to Sections~\ref{section:theorems} and~\ref{section:lemmas}, respectively.

\section{Notation and Problem Formulation}\label{section:model}

We begin with formalizing the problem of matching two sets of features $\{X_1,\ldots,X_n\}$ and $\{\Xdiese_1,\ldots,\Xdiese_m\}$ with $n,m\ge 2$.
In what follows we assume that the observed features are randomly generated from the model
\begin{equation}\label{model}
\begin{cases}
X_i = \t_i + \s_i\xi_i \ , \\
\Xdiese_j = \tdiese_j + \sdiese_j\xidiese_j,
\end{cases}\quad i=1,\ldots,n \text{ and } j=1,\ldots,m
\end{equation}
where
\begin{itemize}\setlength{\itemsep}{-2pt}
\item $\btheta=\{\t_1,\ldots,\t_n\}$ and $\btdiese=\{\tdiese_1,\ldots,\tdiese_m\}$ are two collections of vectors from $\RR^d$, corresponding to
the original features, which are unavailable,
\item $\s_1,\ldots,\s_n,\sdiese_1,\ldots,\sdiese_m$ are positive real numbers corresponding to the levels of noise contaminating each feature,
\item $\xi_1,\ldots,\xi_n$ and $\xidiese_1,\ldots,\xidiese_m$ are two independent sets of i.i.d.\ random vectors drawn from
    the Gaussian distribution with zero mean and identity covariance matrix.
\end{itemize}

The task of feature matching consists in finding a bijection $\pi^*$ between the largest possible subsets $S_1$ and $S_2$ of $\{1,\ldots,n\}$ and $\{1,\ldots,m\}$ respectively, such that
\begin{equation}\label{pistar}
\forall\ i\in S_2,\quad \tdiese_i \equiv \t_{\pi^*(i)},
\end{equation}
where $\equiv$ is an equivalence relation that we call \textit{matching criterion}. The features that do not belong to $S_1$ or $S_2$ are called \textit{outliers}. To ease presentation, we mainly focus on the case where the matching criterion is the equality of two vectors. However, as discussed in Section~\ref{subsection5.2}, most results carry over the equivalence corresponding to equality of two vectors transformed by a given linear transformation.
Furthermore, it turns out that statistical inference for the matching problem is already quite involved when no outlier is present in the data. Therefore,
we make the following assumption
\begin{equation}\label{m=n}
m=n \qquad\text{and}\qquad S_1=S_2=\{1,\ldots,n\}.
\end{equation}
Note however that the procedures we consider below admit natural counterparts in the setting with outliers. We will also restrict ourselves to
noise levels satisfying some constraints. The two types of constraints we consider, referred to as homoscedastic and heteroscedastic setting, correspond to
the relations $\s_i=\sdiese_i=\s$, $\forall i=1,\ldots,n$, and $\s_{\pi^*(i)}=\sdiese_i$, $\forall i=1,\ldots,n$.

In this formulation, the data generating distribution is defined by the (unknown) parameters $\btheta$,
$\bsigma = (\s_i,\ldots,\s_n)$ and $\pi^*$.
In the problem of matching, we focus our attention on the problem of estimating the parameter $\pi^*$ only,
considering $\btheta$ and $\bsigma$ as nuisance parameters. In what follows, we denote by  $\prob_{\btheta,\bsigma,\pi^*}$ the probability distribution
of the vector $(X_1,\ldots,X_n,\Xdiese_1,\ldots,\Xdiese_n)$ defined by (\ref{model}) under the conditions
(\ref{pistar}) and (\ref{m=n}). We write $\esp_{\btheta,\bsigma,\pi^*}$ for the expectation
with respect to $\prob_{\btheta,\bsigma,\pi^*}$. The symmetric group, \ie the set of all permutations
of $\{1,\ldots,n\}$, will be denoted by $\Sn$.

We will use two measures of quality for quantifying the error of an estimator $\hat\pi$ of the permutation $\pi^*$. These
errors are defined as the 0\,-1 distance and the normalized Hamming distance between $\hat\pi$ and $\pi^*$, given by
\begin{equation}\label{dist}
\delta_{\text{0\,-1}}(\hat\pi,\pi^*)\triangleq\fcar_{\{\hat\pi\not=\pi^*\}},\qquad
\delta_H(\hat\pi,\pi^*)\triangleq\frac1n\sum_{k=1}^n \fcar_{\{\hat\pi(k)\not=\pi^*(k)\}}.
\end{equation}
Our ultimate goal is to design estimators that have an expected error smaller than a prescribed level $\alpha$
under the weakest possible conditions on the nuisance parameter $\btheta$. The estimation of the permutation or, equivalently,
the problem of matching is more difficult when the features are hardly distinguishable. To quantify this phenomenon, we introduce
the separation distance $\k(\btheta)$ and the relative separation distance $\bar\k(\btheta,\bsigma)$, which measure the minimal distance
between distinct features and the minimal distance-to-noise ratio, respectively. The precise definitions are
\begin{equation}
\k(\btheta) \triangleq \min_{i\neq j} \|\t_i-\t_j\|,\qquad
\bar\k(\btheta,\bsigma) \triangleq \min_{i\neq j} \frac{\|\t_i-\t_j\|}{(\s_i^2+\s_j^2)^{1/2}}. \label{kappa}
\end{equation}
We will see that in the heteroscedastic case the last quantity is more suitable for characterizing the behavior of the estimators
than the first one.

Clearly, if $\bar\k(\btheta,\bsigma)=0$, then the parameter $\pi^*$ is nonidentifiable, in the sense that there exist
two different permutations $\pi^*_1$ and $\pi^*_2$ such that the distributions $\prob_{\btheta,\bsigma,\pi^*_1}$
and $\prob_{\btheta,\bsigma,\pi^*_2}$ coincide. Therefore, the condition $\bar\k(\btheta,\bsigma)>0$ is necessary
for the existence of consistent estimators of $\pi^*$. Furthermore, good estimators are those consistently estimating $\pi^*$
even if $\bar\k(\btheta,\bsigma)$ is small. To give a precise sense to these considerations, let $\a\in(0,1)$ be a prescribed
tolerance level and let us call \textit{perceivable separation distance} of a given estimation procedure $\hat\pi$ the quantity
\begin{equation*}
\bar\k_\a(\hat\pi) \triangleq \inf\Big\{\k>0 \suchthat \max_{\pi\in\Sn} \sup_{\bar\k(\btheta,\bsigma)>\k} \prob_{\btheta,\bsigma,\pi} (\hat\pi\neq\pi)\le\a\Big\},
\end{equation*}
where we skip the dependence on $n$, $d$, $\bsigma$ and $\bsigmadiese$. Here, the perceivable separation distance
is defined with respect to the 0\,-1 distance, the corresponding definition for the Hamming distance is obtained
by replacing $\prob_{\btheta,\bsigma,\pi} (\hat\pi\neq\pi)$ by $\esp_{\btheta,\bsigma,\pi} [\delta_H(\hat\pi,\pi)]$.
Finally, we call \textit{minimax separation distance} the smallest possible {perceivable separation distance} achieved by an estimator
$\hat\pi$, \ie
\begin{equation}\label{mmx}
\bar\k_\a \triangleq \inf_{\hat\pi}~\bar\k_\a(\hat\pi),
\end{equation}
where the infimum is taken over all possible estimators of $\pi^*$. In the following sections, we establish nonasymptotic upper and lower bounds on
the minimax separation distance which coincide up to a multiplicative constant independent of $n$, $d$, $\bsigma$ and $\bsigmadiese$. We also
show that a suitable version of the maximum profiled likelihood estimator is minimax-separation-rate-optimal both in the homoscedastic and heteroscedastic
settings.

\section{Estimation Procedures}\label{section:estimators}


As already mentioned, we will consider four estimators. The simplest one, called greedy algorithm and
denoted by $\pi^{\rm gr}$ is defined as follows: $\pi^{\rm gr}(1) = \arg\min_{j\in\{1,\ldots,n\}} \|X_j-\Xdiese_1\|$
and, for every $i\in\{ 2,\ldots,n\}$, recursively define
\begin{align}\label{greedy}
\pi^{\rm gr}(i)\triangleq \arg\min_{j\not\in\{\pi^{\rm gr}(1),\ldots,\pi^{\rm gr}(i-1)\}} \|X_j-\Xdiese_i\|.
\end{align}
A drawback of this estimator is that it is not symmetric: the resulting permutation depends on the initial numbering of the
features. However, we will show that this estimator is minimax-separation-rate-optimal in the homoscedastic setting.


A common approach for avoiding incremental estimation and taking into consideration all the observations at the same time
consists in defining the estimator $\hat\pi$ as a maximizer of the profiled likelihood. In the homoscedastic case, which is
the first setting we studied, the computations lead to the estimator
\begin{equation}
\pi^{\textrm{LSS}}  \triangleq \arg\min_{\pi\in\Sn} \sum_{i=1}^n \|X_{\pi(i)}-\Xdiese_i\|^2. \label{LSS}
\end{equation}
which will be referred to as the \textit{Least Sum of Squares} (LSS) estimator.

The LSS estimator takes into account the distance between the observations irrespectively of the noise levels.
The fact of neglecting the noise levels, while harmless in the homoscedastic setting, turns out to cause serious
loss of efficiency in terms of the perceivable distance of separation in the setting of heteroscedastic noise.
Yet, in the latter setting, the distance between the observations is not as relevant as the signal-to-noise ratio
${\|\t_i-\t_j\|^2}/({\s_i^2+\s_j^2})$. Indeed, when the noise levels are small, two noisy but distinct vectors
are easier to distinguish than when the noise levels are large.

The computation of the maximum likelihood estimator in the heteroscedastic case with known noise levels also suggests that
the signal-to-noise ratio should be taken into account. In this setting, the likelihood maximization leads to the
\textit{Least Sum of Normalized Squares} (LSNS) estimator
\begin{equation}
\pi^{\textrm{LSNS}} \triangleq \arg\min_{\pi\in\Sn}\sum_{i=1}^n \frac{\|X_{\pi(i)}-\Xdiese_i\|^2}{\s_{\pi(i)}^2+{\sdiese_i}^2}. \label{LSNS}
\end{equation}
We will often call the LSNS a pseudo-estimator, to underline the fact that it requires the knowledge of the noise levels $\s_i$ which
are generally unavailable.

In the general setting, when no information on the noise levels is available, the likelihood is maximized over all nuisance parameters (features and noise levels). But this problem is underconstrained, and the result of this maximization is $\pinf$. This can be circumvented by assuming a proper relation between the noise levels. As mentioned earlier, we chose the  assumption
\begin{equation}\label{sigma_h}
\forall\ i\in\{1,\ldots,n\},\quad \sdiese_i = \s_{\pi^*(i)}.
\end{equation}
The maximum likelihood estimator under this constraint is the \textit{Least Sum of Logarithms} (LSL) defined as
\begin{equation}
\pi^{\textrm{LSL}}  \triangleq \arg\min_{\pi\in\Sn} \sum_{i=1}^n \log \|X_{\pi(i)}-\Xdiese_i\|^2. \label{LSL}
\end{equation}
We will prove that this estimator is minimax-rate-optimal both in the homoscedastic and the heteroscedastic cases.

\section{Performance of the Estimators}\label{section:performance}

The purpose of this section is to assess the quality of the aforementioned procedures. To this end, we present
conditions for the consistency of these estimators in the form of upper bounds on their perceivable separation
distance. Furthermore, to compare this bounds with the minimax separation distance, we establish lower bounds on the latter
and prove that it coincides up to a constant factor with the perceivable separation  distance of the LSL estimator.

\subsection{Homoscedastic Setup}
We start by considering the homoscedastic case, in which upper and lower bounds matching up to a constant are obtained for all the estimators
introduced in the previous section.

\begin{theorem}\label{upperbound1}
Let $\a\in (0,1)$ be a tolerance level and $\s_j=\sdiese_j=\s$ for all $j\in\{1,\ldots,n\}$.
If $\hat\pi$ denotes any one of the estimators (\ref{greedy})-(\ref{LSL}), then
\begin{equation*}
\bar\k_\a(\hat\pi)\le 4 \max\Big\{\Big(2\log \frac{8n^2}{\a}\Big)^{1/2}, \Big(d\log\frac{4n^2}{\a}\Big)^{1/4} \Big\}.
\end{equation*}
\end{theorem}

An equivalent way of stating this result is that if
\begin{equation*}
\k = 4\s \max\Big\{\Big(4\log \frac{8n^2}{\a}\Big)^{1/2}, \Big(4d\log\frac{4n^2}{\a}\Big)^{1/4} \Big\}
\end{equation*}
and $\T_\k$ is the set of all $\btheta\in\RR^{n\times d}$ such that $\k(\btheta)\ge \k$, then
\begin{equation*}
\max_{\pi^*\in\Sn}\sup_{\btheta\in\T_\k}\prob_{\btheta,\bsigma,\pi^*}(\hat\pi\neq\pi^*) \le \a
\end{equation*}
for all the estimators defined in Section~\ref{section:estimators}.
Note that this result is nonasymptotic. Furthermore, it tells us that the perceivable separation distance of the
procedures under consideration is at most of the order of
\begin{equation}
\max\Big\{(\log n)^{1/2}, (d \log n)^{1/4}\Big\}.
\end{equation}
It is interesting to observe that there are two regimes in this rate, the boundary of which corresponds to the case where $d$ is of the order of $\log n$.
For dimensions that are significantly smaller than $\log n$, the perceivable distance of separation is dimension free. On the other hand,
when $d$ is larger than $c\log n$ for some absolute constant $c>0$, the perceivable distance of separation deteriorates with increasing
$d$ at the polynomial rate $d^{1/4}$. However, this result does not allow us to deduce any hierarchy between the four estimators, since it
provides the same upper bound for all of them. Moreover, as stated in the next theorem, this bound is optimal up to a multiplicative constant.

\begin{theorem}\label{lowerbound1}
Assume that $n\ge6$, $\T_\k$ is the set of all $\btheta\in\RR^{n\times d}$ such that $\k(\btheta)\ge \k$.
Then there exist two absolute constants $c,C>0$ such that for
\begin{equation*}
\k=2^{-5/2}{\s}\max\Big\{(\log n)^{1/2}, (c d \log n)^{1/4}\Big\},
\end{equation*}
the following lower bound holds
\begin{equation*}
\inf_{\hat\pi} \max_{\pi^*\in\Sn} \sup_{\btheta\in\T_\k}\prob_{\btheta,\bsigma,\pi^*}\big(\hat\pi\neq\pi^*\big) >  C,
\end{equation*}
where the infimum is taken over all permutation estimators.
\end{theorem}

An equivalent way of stating this result is to say that the minimax distance of separation satisfies the inequality
$$
\bar\k_\a \ge 2^{-5/2}\max\big\{(\log n)^{1/2}, (c d \log n)^{1/4}\big\}.
$$
Combined with Theorem~\ref{upperbound1}, this implies that the minimax rate of separation is given by the expression
$\max\big\{(\log n)^{1/2}, (d \log n)^{1/4}\big\}$.

In order to avoid any possible confusion, we emphasize that the rate obtained in this and subsequent sections concerns the
speed of decay of the separation distance and not the estimation risk measured by $\prob_{\btheta,\bsigma,\pi^*}\big(\hat\pi\neq\pi^*\big)$.
For the latter, considering $\kappa$ as fixed, one readily derives from Theorem~\ref{upperbound1} that
\begin{align}\label{rem:a}
\max_{\pi^*\in\Sn} \sup_{\btheta\in\T_\k}\prob_{\btheta,\bsigma,\pi^*}\big(\hat\pi\neq\pi^*\big)
\le \max\Big\{8n^2\exp\Big(-\frac{\kappa^2}{2^{6}\sigma^2}\Big),4n^2\exp\Big(-\frac{\kappa^4}{2^{10}d\sigma^4}\Big)\Big\}
\end{align}
for the four estimators $\hat\pi$ defined in the previous section by equations (\ref{greedy})-(\ref{LSL}). We do not know whether the right-hand side of this
inequality is the correct rate for the minimax risk $R^{\rm mmx} =\inf_{\hat\pi}  \sup_{(\btheta,\pi^*)\in\T_\k\times\Sn}
\prob_{\btheta,\bsigma,\pi^*}\big(\hat\pi\neq\pi^*\big)$. In fact, one can adapt the proof of Theorem~\ref{lowerbound1}
to get a lower bound on  $R^{\rm mmx}$ which is of the same form as the right-hand side
in (\ref{rem:a}), but with constants $2^6$ and $2^{10}$ replaced by smaller ones. The ratio of such a lower bound and the upper bound in
(\ref{rem:a}) tends to 0 and, therefore, does not provide the minimax rate of estimation, in the most common sense of the
term. However, one may note that the minimax rate of separation established in this work is the analogue of the
minimax rate of estimation of the Bahadur risk, see \cite{Bahadur,Korostelev,KorostelevSpok}.

\subsection{Heteroscedastic Setup}

We switch now to the heteroscedastic setting, which allows us to discriminate between the four procedures.
Note that the greedy algorithm, the LSS and the LSL have a serious advantage over the LSNS since they can
be computed without knowing the noise levels $\bsigma$.

\begin{theorem}\label{upperbound2}
Let $\a\in(0,1)$ and condition (\ref{sigma_h}) be fulfilled. If $\hat\pi$ is either $\pi^{\textrm{LSNS}}$ (if the noise
levels $\s_i,\sdiese_i$ are known) or $\pi^{\textrm{LSL}}$ (when the noise levels are unknown),
then
\begin{equation*}
\bar\k_\a(\hat\pi) \leq 4\max\Big\{\Big(2\log \frac{8n^2}{\a}\Big)^{1/2}, \Big(d\log\frac{4n^2}{\a}\Big)^{1/4} \Big\}.
\end{equation*}
\end{theorem}

This result tells us that the performance of the LSNS and LSL estimators, measured in terms of the order of magnitude of the
separation distance, is not affected by the heteroscedasticity of the noise levels. Two natural questions arise: 1) is this performance the best possible
over all the estimators of $\pi^*$ and 2) is the performance of the LSS and the greedy estimator as good as that of
the LSNS  and LSL?

To answer the first question, we should start by adapting the notion of minimax separation distance to the case of unknown
noise levels. Indeed, the definition (\ref{mmx}) given in previous sections involves a minimum over all possible estimators
$\hat\pi$ which are allowed to depend on $\bsigma$. On the one hand, considering $\bsigma$ as known limits considerably the
scope of applications of the methods. On the other hand, from a theoretical point of view, knowing $\bsigma$ may lead to
a substantially better minimax rate and even to a separation equal to zero, which correspond to an estimator that has no real practical
interest.  Indeed, under condition~(\ref{sigma_h}), it is possible to estimate the permutation $\pi^*$ by only considering the
noise levels, without any assumption on the distance between the features. For instance, we can define an estimator $\hat\pi$ as follows.
We set $\hat\pi(1) = \arg\min_{i=1,\ldots,n} \big| \frac1{2d}{\|X_i-\Xdiese_1\|^2} - \s_i^2 \big|$
and recursively, for every $j\in\{ 2,\ldots,n\}$,
\begin{equation*}
\hat\pi(j) = \arg\min_{i\not\in\{\hat\pi(1),\ldots,\hat\pi(j-1)\}} \Big| \frac1{2d}{\|X_i-\Xdiese_j\|^2} - \s_i^2 \Big|.
\end{equation*}
This leads to an accurate estimator of $\pi^*$---for vectors $\{\theta_j\}$ that are very close---as soon as the noise
levels are different enough from each other.
In particular, the estimated permutation $\hat\pi$ coincides with the true permutation $\pi^*$ on the event
\begin{equation}\label{eq:14}
\forall i\in\{1,\ldots,n\},\qquad \Big|\frac1{2d}{\|X_{\pi^*(i)}-\Xdiese_i\|^2} - \s_{\pi^*(i)}^2 \Big| < \min_{j\neq \pi^*(i)} \Big|  \frac{\|X_j-\Xdiese_i\|^2}{2d} - \s_j^2\Big|.
\end{equation}
Using standard bounds on the tails of the $\chi^2$ distribution recalled in Lemma~\ref{concentration} of Section~\ref{section:theorems},
in the case when all the vectors $\theta_j$ are equal, one can check that the left-hand side in (\ref{eq:14}) is of the order
of $\s_{\pi^*(i)}^2\sqrt{(\log n)/d}$ while the right-hand side is at least of the order of
$\frac12\min_{j\not = \pi^*(i)} |\s_j^2-\s_{\pi^*(i)}^2|- C(\sigma_{\pi^*(i)}^2+\s_{j}^2)\sqrt{(\log n)/d}$. This implies that we can consistently identify the permutation when
\begin{equation*}
\forall (i,j)\in\{1,\ldots,n\}^2,
\quad i\not= j,\qquad \Big|\frac{\s_j^2}{\s_i^2}-1\Big| \ \gg \ \sqrt{\frac{\log n}{d}},
\end{equation*}
even if the separation distance is equal to zero. In order to discard such kind of estimators from the competition in the procedure of determining the minimax rates, we restrict our attention to the noise levels for which the values $|\frac{\s_j^2}{\s_i^2}-1|$ are not
larger than $C\sqrt{(\log n)/d}$ for $j\not=i$.

\begin{theorem}\label{lowerbound2}
Assume that $n\ge6$, $\bar\T_\k$ is the set of all $\btheta\in\RR^{n\times d}$ such that $\bar\k(\btheta,\bsigma)\ge \k$ and
\begin{equation*}
\frac{\max_i\s^2_i}{\min_i\s^2_i} -1 \le \frac14 \ \sqrt{\frac{\log n}{d}}.
\end{equation*}
Then there exist two constants $c,C>0$ such that
$\k<(1/8) \max\{ (\log n)^{1/2},  (cd\log n)^{1/4} \}$,
implies that
\begin{equation*}
\inf_{\hat\pi} \max_{\pi^*\in\Sn} \sup_{\btheta\in\bar\T_{\k}} \prob_{\btheta,\bsigma,\pi^*} (\hat\pi\neq\pi^*) > C,
\end{equation*}
where the infimum is taken over all permutation estimators.
\end{theorem}

It is clear that the constants $c$ and $C$ of the previous theorem are closely related. The inspection of the proof shows that, for instance,
if $c\le 1/20$ then  $C$ is larger than $17\%$.

Let us discuss now the second question raised earlier in this section and concerning the theoretical properties
of the greedy algorithm and the LSS under heteroscedasticity. In fact, the perceivable distances of separation
of these two procedures are significantly worse than those of the LSNS and the LSL especially for large
dimensions $d$.  We state the corresponding result for the greedy algorithm, a similar conclusion
being true for the LSS as well. The superiority of the LSNS and LSL is also confirmed by the numerical
simulations presented in Section~\ref{section:simulations} below. In the next theorem and in the sequel of the paper,
we denote by $\id$ the identity permutation defined by $\id(i) = i$ for all $i\in\{1,\ldots,n\}$.

\begin{theorem}\label{lowerbound3}
Assume that $d\ge 225\log 6$, $n=2$, $\s_1^2 =3$ and $\s_2^2=1$.
Then the condition  $\k < 0.1(2d)^{1/2}$ implies that
$$
\sup_{\btheta\in\bar\T_\k}\prob_{\btheta,\bsigma,\id}(\pi^{\textrm{gr}}\neq\id) \ge 1/2.
$$
\end{theorem}

This theorem shows that if $d$ is large, the necessary condition for $\pi^{\textrm{gr}}$ to be
consistent is much stronger than the one obtained for $\pi^{\textrm{LSL}}$ in Theorem~\ref{upperbound2}. Indeed,
for the consistency of $\pi^{\textrm{gr}}$, $\k$ needs to be at least of the order of $d^{1/2}$, whereas
$d^{1/4}$ is sufficient for the consistency of $\pi^{\textrm{LSL}}$. Hence, the maximum likelihood estimators LSNS and LSL
that take into account noise heteroscedasticity are, as expected, more interesting than the simple greedy estimator\footnote{It 
should be noted that the conditions of Theorem~\ref{lowerbound3} are not compatible with those of Theorem~\ref{lowerbound2}. 
Hence, strictly speaking, the former does not imply that the greedy estimator is not minimax under the conditions of the latter.}.


\section{Computational Aspects}\label{section:computationalaspects}


At first sight, the computation of the estimators (\ref{LSS})-(\ref{LSL}) requires to
perform an exhaustive search over the set of all possible permutations, the number of which, $n!$, is
prohibitively large. This is in practice impossible to do on a standard PC as soon as $n \geq 20$.
In this section, we show how to compute these (maximum likelihood) estimators in polynomial time using, for instance,
algorithms of linear programming\footnote{The idea of reducing the problem of permutation estimation to a linear program
has been already used in the literature, however, without sufficient theoretical justification: see, for instance,
\citet{Jebara2003}.}.

To explain the argument, let us consider the LSS estimator
\begin{equation*}
\pi^{{\textrm{LSS}}} = \arg\min_{\pi\in\Sn} \sum_{i=1}^n \|X_{\pi(i)}-\Xdiese_i\|^2.
\end{equation*}
For every permutation $\pi$, we denote by $P^\pi$ the $n\times n$ permutation matrix with coefficients
$P^\pi_{ij} = \fcar_{\{ j=\pi(i) \}}$. Then we can give the equivalent formulation
\begin{equation}\label{discrete}
\pi^{\textrm{LSS}} = \arg\min_{\pi \in\Sn} \trace \big(M P^\pi\big),
\end{equation}
where $M$ is the matrix with coefficient $\|X_i-\Xdiese_j\|^2$ at the $i^{th}$ row and $j^{th}$ column. The cornerstone of our next argument is the Birkhoff-von Neumann theorem stated below, which can be found for example in \citep{BudishCheKojimaMilgrom2009}.

\begin{theorem}[Birkhoff-von Neumann Theorem]\label{birkhoff-von-neumann}
Assume that $\mathcal{P}$ is the set of all doubly stochastic matrices of size $n$, \ie the matrices whose
entries are nonnegative and sum up to $1$ in every row and every column. Then every matrix in $\mathcal{P}$
is a convex combination of matrices $\{P^\pi : \pi\in\Sn\}$. Furthermore, permutation matrices are the vertices
of the simplex $\mathcal P$.
\end{theorem}

In view of this result, the combinatorial optimization problem (\ref{discrete}) is equivalent to
the following problem of continuous optimization:
\begin{equation}\label{cont}
P^{\textrm{LSS}} = \arg\min_{P \in\mathcal P} \trace \big(M P\big),
\end{equation}
in the sense that $\pi$ is a solution to (\ref{discrete}) if and only if $P^\pi$ is a solution to (\ref{cont}).
To prove this claim, let us remark that for every $P\in\mathcal{P}$, there exist coefficients $\a_1,\ldots,\a_{n!}\in[0,1]$ such that
$P = \sum_{i=1}^{n!} \a_i P^{\pi_i}$ and $\sum_{i=1}^{n!} \a_i = 1$.
Therefore, we have $\trace\big(MP\big) = \sum_{i=1}^{n!} \a_i \trace\big(MP^{\pi_i}\big) \ge \min_{\pi\in\Sn} \trace\big(MP^{\pi}\big)$ and $\trace\big(MP^{\textrm{LSS}}\big)\ge \trace\big(MP^{\pi^{\textrm{LSS}}}\big).$

The great advantage of (\ref{cont}) is that it concerns the minimization of a linear function under linear constraints and, therefore,
is a problem of linear programming that can be efficiently solved even for large values of $n$. The same arguments apply to the estimators $\pi^{\textrm{LSNS}}$ and $\pi^{\textrm{LSL}}$, only the matrix $M$ needs to be changed.

There is a second way to compute the estimators LSS, LSNS and LSL efficiently. Indeed, the computation of the aforementioned maximum likelihood estimators is a particular case of the assignment problem, which consists in finding a minimum weight matching in a weighted bipartite graph, where the matrix of the costs is the matrix $M$ from above. This means that the cost of assigning the $i^{\text{th}}$ feature of the first image to the $j^{\text{th}}$ feature of the second image is either
\begin{itemize}\setlength{\itemsep}{0pt}
\item the squared distance  ${\|X_i-\Xdiese_j\|^2}$,
\item or the normalized squared distance $\|X_i-\Xdiese_j\|^2/(\s_i^2+(\sdiese_j)^2)$,
\item or the logarithm of the squared distance $\log \|X_i-\Xdiese_j\|^2$.
\end{itemize}
The so-called Hungarian algorithm presented in \citet*{Kuhn1955} solves the assignment problem in time $O(n^3)$.

\section{Extensions}\label{section:extensions}

In this section, we briefly discuss possible extensions of the foregoing results to
other distances, more general matching criteria and to the estimation of an arrangement.

\subsection{Minimax Rates for the Hamming Distance}
In the previous sections, the minimax rates were obtained for the error of estimation measured by the risk $\prob_{\btheta,\bsigma\pi^*}(\hat\pi\neq\pi^*)
=\esp_{\btheta,\bsigma,\pi^*}[\delta_{\text{0\,-1}}(\hat\pi,\pi^*)]$, which may be considered as too restrictive. Indeed, one could find acceptable an estimate
having a small number of mismatches, if it makes it possible to significantly reduce the perceivable distance of separation. These considerations lead
to investigating the behavior of the estimators in terms of the Hamming loss, \textit{i.e.}, to studying the risk
\begin{equation*}
\esp_{\btheta,\bsigma,\pi^*}[\delta_{H}(\hat\pi,\pi^*)]=\esp_{\btheta,\bsigma,\pi^*}
\Big[\frac1n \sum\limits_{i=1}^n \fcar_{\{\hat\pi(i)\neq\pi^*(i)\}}\Big]
\end{equation*}
corresponding to the expected average number of mismatched features. Another advantage of studying the Hamming loss instead of 
the 0\,-1 loss is that
the former sharpens the difference between the performances of various estimators. Note, however, that thanks to the inequality
$\delta_{H}(\hat\pi,\pi^*)\le \delta_{\text{0\,-1}}(\hat\pi,\pi^*)$, all the upper bounds established for the minimax rate of separation
under the 0-1 loss directly carry over to the case of the Hamming loss.
This translates into the following theorem.

\begin{theorem}\label{upperbound4}
Let $\a\in(0,1)$ and condition (\ref{sigma_h}) be fulfilled. If $\hat\pi$ is either $\pi^{\textrm{LSNS}}$
or $\pi^{\textrm{LSL}}$,
then $\bar\k_\a(\hat\pi) \leq 4\max\big\{\big(2\log \frac{8n^2}{\a}\big)^{1/2}, \big(d\log\frac{4n^2}{\a}\big)^{1/4} \big\}$. That is,
if
$$\kappa=4\max\Big\{\Big(2\log \frac{8n^2}{\a}\Big)^{1/2}, \Big(d\log\frac{4n^2}{\a}\Big)^{1/4} \Big\}
$$
and $\bar\T_\k$ is the set of all $\btheta\in\RR^{n\times d}$ such that $\bar\k(\btheta,\sigma)\ge \k$, then
\begin{equation*}
\max_{\pi^*\in\Sn}\sup_{\btheta\in\bar\T_\k}\esp_{\btheta,\bsigma,\pi^*}[\delta_{H}(\hat\pi,\pi^*)] \le \a.
\end{equation*}
\end{theorem}

\begin{table}
\begin{center}
\begin{tabular}{r|rrrrrrrrr}
\toprule
$n=$ & 4 & 5 & 6 & 7 & 8 & 9 & 10 & 11 & 12\\
$M_n \ge $ & 19 & 57 & 179 & 594 & 1939 & 3441 & 11680 & 39520 & 86575
\\
$\frac{\log M_n}{n\log n}\ge$  & 0.53 & 0.50 & 0.48 & 0.47 & 0.455 & 0.412 & 0.407 & 0.401 & 0.381
\\
$\frac{\log M_n}{n\log n}\le$  & 0.53 & 0.50 & 0.48 & 0.47 & 0.455 & 0.445 & 0.436 & 0.427 & 0.420
\\
\bottomrule
\end{tabular}
\caption{The values of $M_n=\mathcal M(1/4, B_{2,n}(2),\delta_H)$ for $n\in\{4,\ldots,12\}$.
The lower bound is just the cardinality of one $\epsilon$-packing, not necessarily the largest one.
The upper bound is merely the cardinality of the $\ell_2$-ball.
\label{tab:1}}
\end{center}
\end{table}
While this upper bound is an immediate consequence of Theorem~\ref{upperbound2}, getting lower bounds for the Hamming loss
appears to be more difficult. To state the corresponding result, let is consider the case of homoscedastic noise and introduce some
notation. We denote by $\delta_2(\cdot,\cdot)$ the normalized $\ell^2$-distance on the space of permutations $\Sn$:
$\delta_2(\pi,\pi')^2=\frac1n\sum_{k=1}^n\big(\pi(k)-\pi'(k)\big)^2$. Let $B_{2,n}(R)$ be the ball of $(\Sn,\delta_2)$
with radius $R$ centered at $\id$. As usual, we denote by $\mathcal M(\e, B_{2,n}(R),\delta_H)$ the $\e$-packing number
of the $\ell_2$-ball $B_{2,n}(R)$ in the metric $\delta_H$. This means that $\mathcal M(\e, B_{2,n}(R),\delta_H)$ is the
largest integer $M$  such that there exist permutations $\pi_1,\ldots,\pi_M\in B_{2,n}(R)$ satisfying $\delta_H(\pi_i,\pi_j)\ge \e$
for every $i\not=j$. One can easily check that replacing $B_{2,n}(R)$ by any other ball of radius $R$ leaves the packing number
$\mathcal M(\e, B_{2,n}(R),\delta_H)$ unchanged. We set $M_n=\mathcal M(1/4, B_{2,n}(2),\delta_H)$.

\begin{theorem}\label{lowerbound4}
Let $\sigma_k=\sigma$ for all $k\in\{1,\ldots,n\}$ and $\bar\T_\k$ is the set of all $\btheta\in\RR^{n\times d}$ such that $\bar\k(\btheta,\bsigma)\ge \k$. Furthermore, assume that one of the following two conditions is fulfilled:
\begin{itemize}
\item $n\ge 3$ and $\k = (1/4)\big(\frac{\log M_n}{n}\big)^{1/2}$,
\item $n\ge 26$, $d\ge 24\log n$ and $\k\le (1/8)(d\log n)^{1/4}$.
\end{itemize}
Then
$\inf_{\hat\pi} \max_{\pi^*\in\Sn} \sup_{\btheta\in\bar\T_\k} \esp_{\btheta,\bsigma,\pi^*} [\delta_H(\hat\pi,\pi^*)] >  2.15\%$.
\end{theorem}

This result implies that in the regime of moderately large dimension, $d\ge 24\log n$, the minimax rate of the separation is the same as
the one under the 0-1 loss and it is achieved by the LSL estimator. The picture is less clear in the regime of small dimensions, $d=o(\log n)$.
If one proves that for some $c>0$, the inequality $\log M_n\ge cn\log n$ holds for every $n\ge 3$, then the lower bound of the last theorem
matches the upper bound of Theorem~\ref{upperbound4} up to constant factors and leads to the minimax rate of separation
$\max\{(\log n)^{1/2},(d\log n)^{1/4}\}$. Unfortunately, we were unable to find any result on the order of magnitude of $\log M_n$,
therefore,  we cannot claim that there is no gap between our lower bound and the upper one. However, we did a small experiment for
evaluating $M_n$ for small values of $n$. The result is reported in Table~\ref{tab:1}.

\subsection{More General Matching Criteria}\label{subsection5.2}

In the previous sections, we were considering two vectors $\t_i$ and $\tdiese_j$ as matching if $\t_i\equiv\tdiese_j$, and $\equiv$
was the usual equality. In this part, we show that our results can be extended to more general matching criteria, defined as follows.
Let $A$, $\Adiese$ be two known $p\times d$ matrices with some $p\in\mathbb N$ and $b$, $\bdiese$ be two known vectors from $\RR^{d}$.
We write $\t\equiv_{A,b} \tdiese$, if
\begin{equation}\label{A,b}
A(\theta-b)=\Adiese(\tdiese -\bdiese).
\end{equation}
Note that the case of equality studied in previous sections is obtained for $A =\Adiese=\mathbf I_d$ and $b=\bdiese=0$, where $\mathbf I_d$ is
the identity matrix of size $d$. Let us first note that without loss of generality, by a simple transformation of the features, one can replace
(\ref{A,b}) by the simpler relation
\begin{equation}\label{A'}
\bar\theta=B \bartdiese,
\end{equation}
where $\bar\t\in\RR^{d_1}$ for $d_1 = \text{rank}(A)$, $\bartdiese\in\RR^{d_2}$ for $d_2 = \text{rank}(\Adiese)$  and
$B$ is a $d_1\times d_2$ known matrix. Indeed, let $A = U^\top\Lambda V$ (resp.\ $\Adiese = \tilde U^\top \tilde\Lambda\tilde V$)
be the singular value decomposition of $A$ (resp.\ $\Adiese$), with orthogonal matrices $U\in\RR^{d_1\times p}$, $V\in\RR^{d_1\times d}$
and a diagonal matrix $\Lambda\in\RR^{d_1\times d_1}$ with positive entries. Then, one can deduce (\ref{A'}) from (\ref{A,b}) by setting
$\bar\t = V(\theta-b)$, $\bartdiese=\tilde V(\tdiese-\bdiese)$ and $B = \Lambda^{-1}U\tilde U^\top\tilde\Lambda$. Of course, the same
transformation should be applied to the observed noisy features, which leads to $\bar X_i= V(X_i-b)$ and
$\barXdiese_i=\tilde V(\Xdiese-\bdiese)$. Since $V$ and $\tilde V$ are orthogonal matrices, \textit{i.e.}, satisfy the relations
$VV^\top = \mathbf I_{d_1}$ and $\tilde V\tilde V^\top = \mathbf I_{d_2}$, the noise component in the transformed noisy features is still white
Gaussian.

All the four estimators introduced in Section~\ref{section:estimators} can be adapted to deal with such type of
criterion. For example, denoting by $M$ the matrix $B(B^\top B)^+B^\top+BB^\top$ where $M^+$ is the
Moore-Penrose pseudoinverse of the matrix $M$, the LSL estimator should be modified as follows
\begin{equation}
\pi^{\textrm{LSL}}  \triangleq \arg\min_{\pi\in\Sn} \sum_{i=1}^n \log \|M^+(\bar X_{\pi(i)}-B\barXdiese_i)\|^2. \label{LSL1}
\end{equation}
All the results presented before can be readily extended to this case. In particular, if we assume that all the nonzero singular values of
$B$ are bounded and bounded away from 0 and $\text{rank}(B)=q$, then the minimax rate of separation is given by
$\max\big\{(\log n)^{1/2}, (q \log n)^{1/4}\big\}$. This rate is achieved, for instance, by the LSL estimator.

Let us briefly mention two situations in which this kind of general affine criterion may be useful. First, if each feature $\theta$ (resp.\
$\tdiese$) corresponds to a patch in an image $I$ (resp.\ $\Idiese$), then for detecting pairs of patches that match each other
it is often useful to neglect the changes in illumination. This may be achieved by means of the criterion $A\theta=A\tdiese$
with $A =  \mathbf I_d-\frac1{d}\mathbf 1_d\mathbf 1_d^\top$. Indeed, the multiplication of the vector $\theta$ by $A$ corresponds
to removing from pixel intensities the mean pixel intensity of the patch. This makes the feature invariant by change of illumination.
The method described above applies to this case and the corresponding rate is $\max\big\{(\log n)^{1/2}, ((d-1)\log n)^{1/4}\big\}$
since the matrix $A$ is of rank $d-1$. Second, consider the case when each feature
combines the local descriptor of an image and the location in the image at which this descriptor is computed. If we have at our
disposal an estimator of the transformation that links the two images and if this transformation is linear, then we are in the
aforementioned framework. For instance, let each $\theta$ be composed of a local descriptor $\mathbf d\in\RR^{d-2}$ and its location
$\mathbf x\in\RR^2$. Assume that the first image $I$ from which the features $\theta_i=[\mathbf d_i,\mathbf x_i]$ are extracted is
obtained from the image $\Idiese$ by a rotation of the plane. Let $R$ be an estimator of this rotation and
$\tdiese_i=[\bddiese_i,\bxdiese_i]$ be the features extracted from the image $\Idiese$. Then, the aim is to find the permutation $\pi$
such that $\bddiese_i=\mathbf d_{\pi(i)}$ and $\bxdiese_i=R\mathbf x_{\pi(i)}$ for every $i=1,\ldots,n$. This corresponds to taking in (\ref{A'})
the matrix $B$ given by
$$
B =\begin{pmatrix}
\mathbf I_{d-2} & 0\\
0 & R
\end{pmatrix}.
$$
This matrix $B$ being orthogonal, the resulting minimax rate of separation is exactly the same as when $B = \mathbf I_d$.

\begin{remark}
An interesting avenue for future research concerns the determination of the minimax rates in the case when the
equivalence of two features is understood under some transformation $A$ which is not completely determined.
For instance, one may consider that the features $\theta$ and $\theta'$ match if there is a matrix $A$
in a given parametric family $\{A_\tau:\tau\in\RR\}\subset \RR^{d\times d}$ for which $\theta = A\theta'$. In
other terms, $\theta\equiv\theta'$ is understood as $\inf_\tau \|\theta-A_\tau\theta'\|=0$. Such a criterion of
matching may be useful for handling various types of invariance (see \cite{CollierDalalyan2011,Collier2012} for
invariance by translation).
\end{remark}

\subsection{Estimation of an Arrangement}\label{sec:arrangement}

An interesting extension concerns the case of the estimation of a general arrangement, \ie the case when $m$
and $n$ are not necessarily identical. In such a situation, without loss of generality, one can assume that
$n\le m$ and look for an injective function $$\pi^*:\{1,\ldots,n\}\to\{1,\ldots,m\}.$$ All the estimators presented
in Section~\ref{section:estimators} admit natural counterparts in this rectangular setting. Furthermore, the
computational method using the Birkhoff-von Neumann theorem is still valid in this setting, and is justified
by the extension of the Birkhoff-von Neumann theorem recently proved by \citet{BudishCheKojimaMilgrom2009}.
In this case, the minimization should be carried out over the set of all matrices $P$ of size $(n,m)$ such
that $P_{i,j}\ge 0,$ and
\begin{equation*}
\begin{cases}
\ \sum_{i=1}^n P_{i,j}\le 1 \\
\ \sum_{j=1}^m P_{i,j}= 1
\end{cases},\quad (i,j)\in\{1,\ldots,n\}\times\{1,\ldots,m\}.
\end{equation*}
From a practical point of view, it is also important to consider the issue of
robustness with respect to the presence of outliers, \ie when for some $i$ there is no
$\Xdiese_j$ matching with $X_i$. The detailed exploration of this problem being out
of scope of the present paper, let us just underline that the LSL-estimator seems
to be well suited for such a situation because of the robustness of the
logarithmic function. Indeed, the correct matches are strongly rewarded because
$\log(0)=-\infty$ and the outliers do not interfere too much with the estimation of
the arrangement thanks to the slow growth of $\log$ in $\pinf$.

\section{Experimental Results}\label{section:simulations}

We have implemented all the procedures in Matlab and carried out numerical experiments on synthetic data.
To simplify, we have used the general-purpose solver SeDuMi~\citep{SeDuMi} for solving linear programs.
We believe that it is possible to speed-up the computations by using more adapted first-order optimization
algorithms, such as coordinate gradient descent. However, even with this simple implementation, the running
times are reasonable: for a problem with $n=500$ features, it takes about six seconds to compute a solution
to (\ref{cont}) on a standard PC.

\begin{figure}[ht!]
{\centerline{\includegraphics[width=0.48\textwidth]{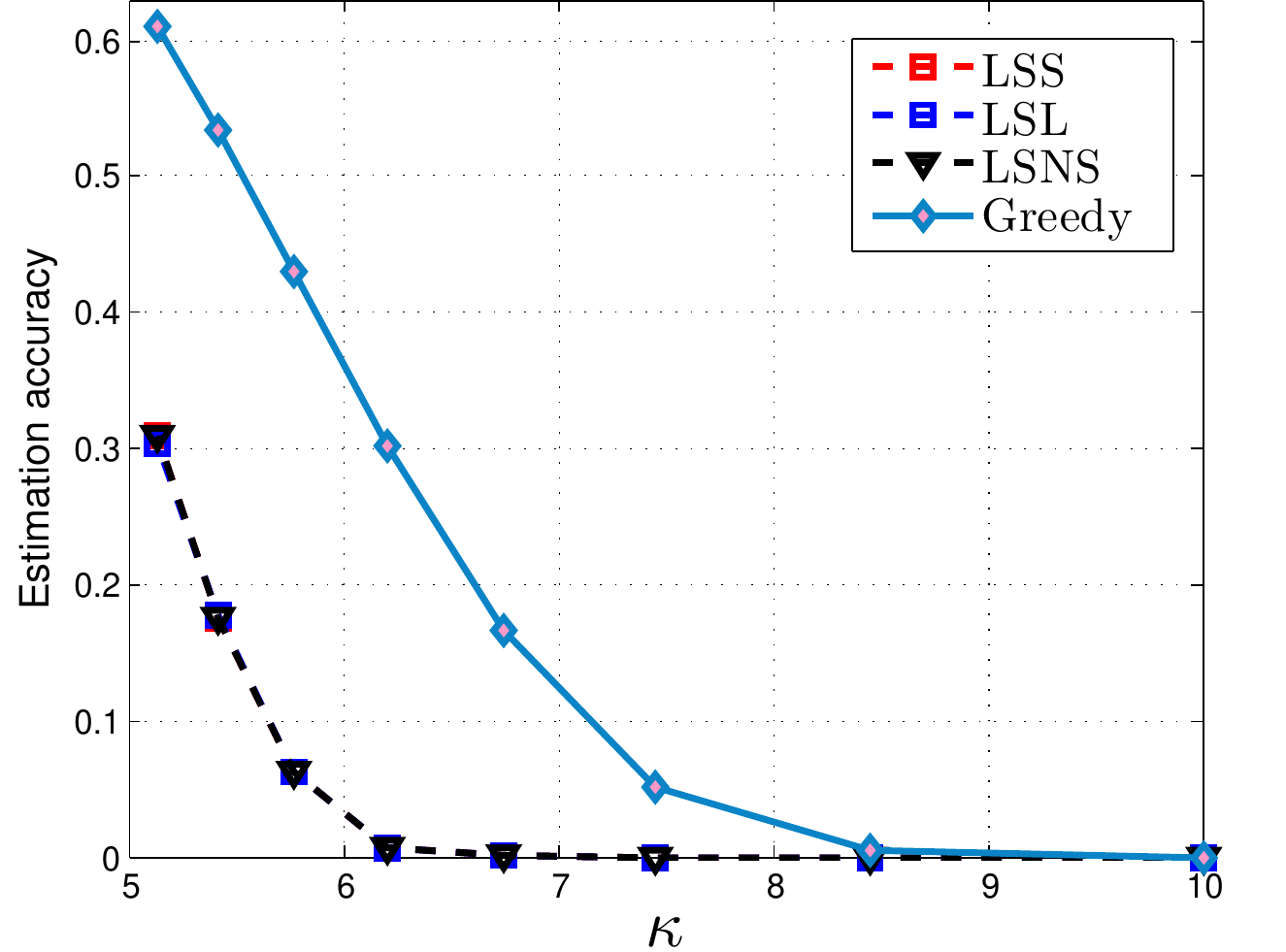}}}
\vspace{-5pt}
\caption{\small Average error rate of the four estimating procedures in the experiment with homoscedastic noise
as a function of the minimal distance $\k$ between distinct features.  One can observe that the LSS, LSNS and LSL
procedures are indistinguishable  and perform much better than the greedy algorithm.\label{fig3:1}
\vspace{-5pt}}
\end{figure}

\begin{figure}[ht!]
\centerline{\includegraphics[width=0.48\textwidth]{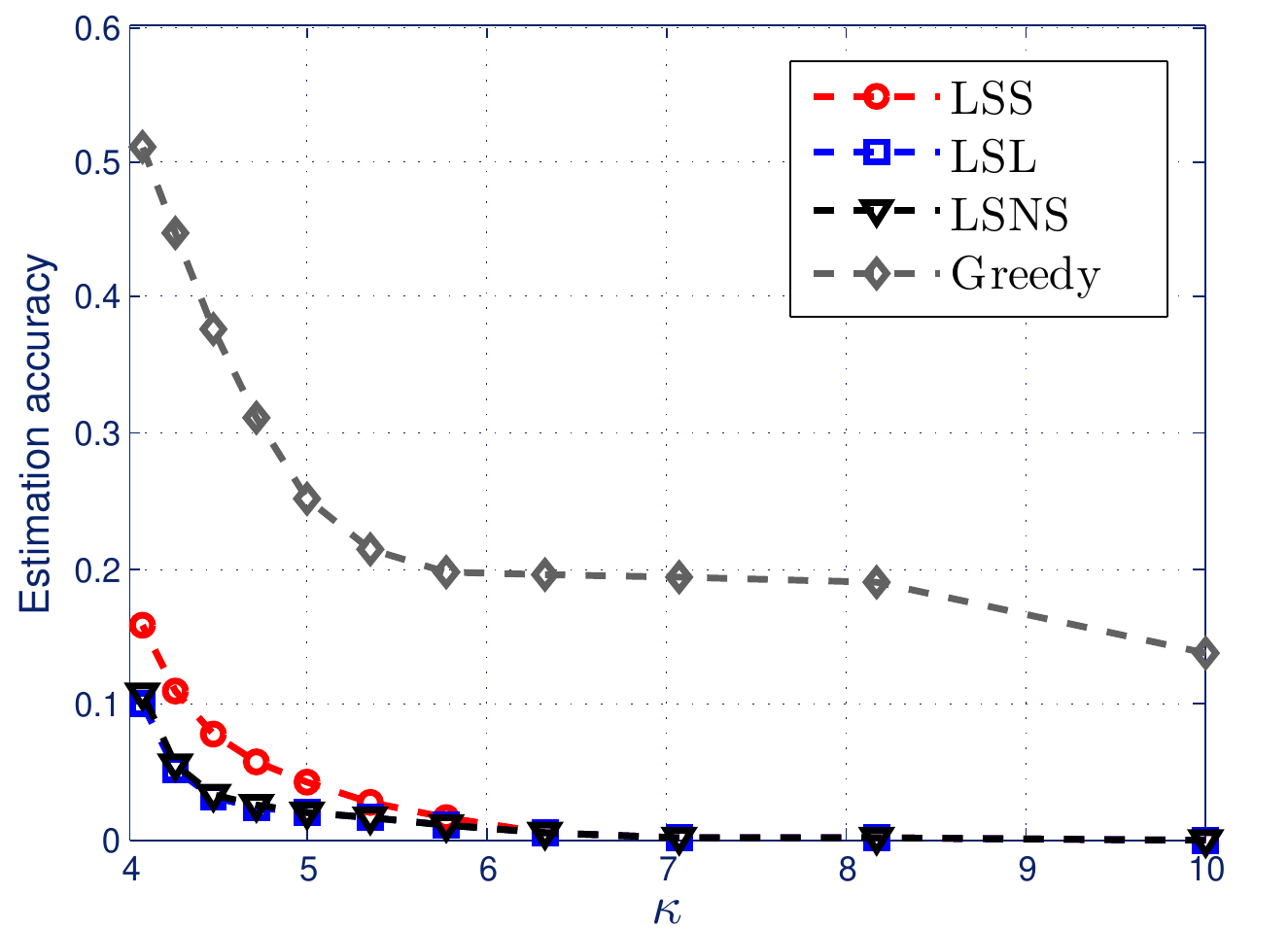}\
\includegraphics[width=0.48\textwidth]{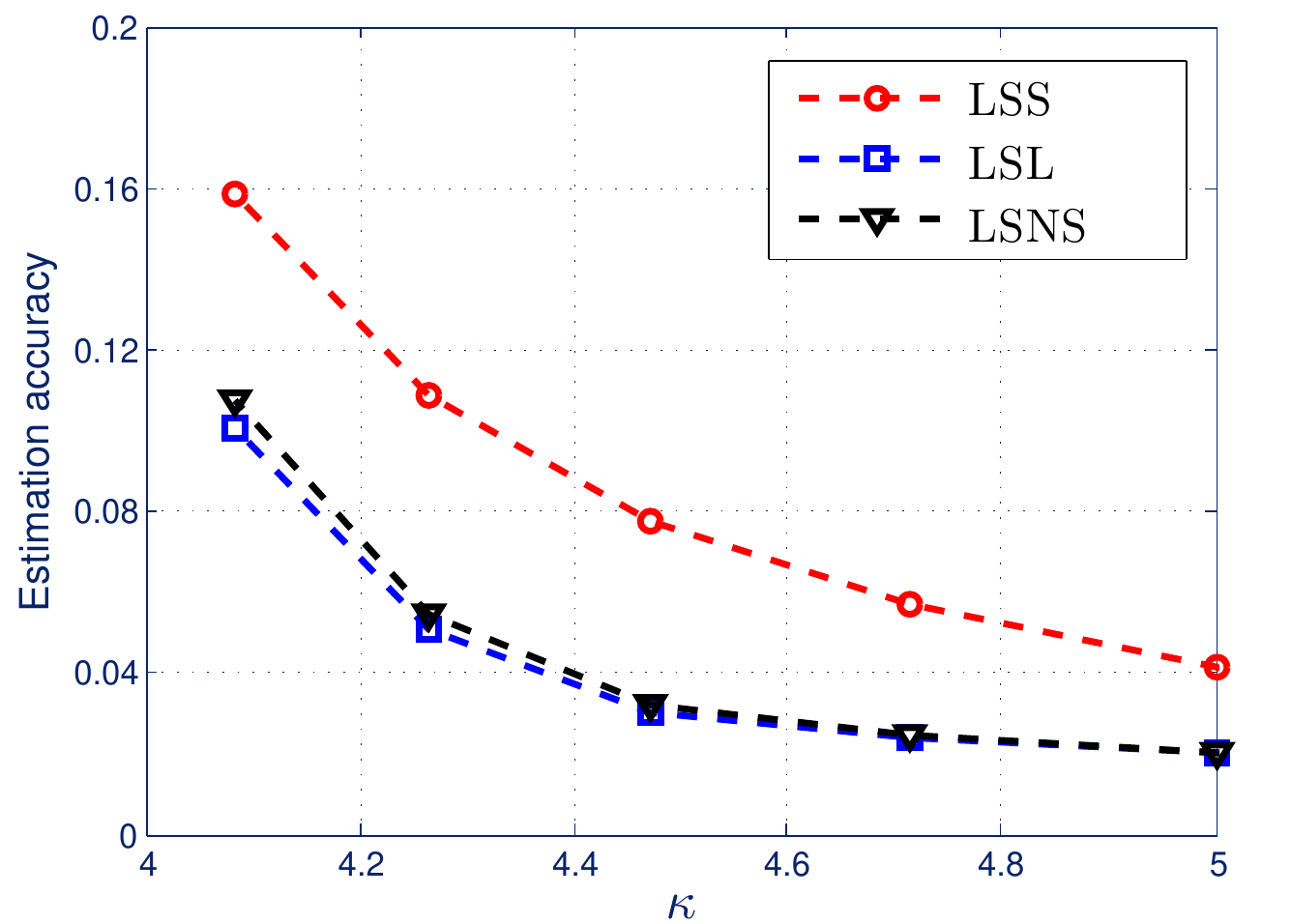}}
\vspace{-5pt}
\caption{\small Left: Average error rate of the four estimating procedures in the experiment with heteroscedastic noise
as a function of the minimal distance $\k$ between distinct features. Right: zoom on the same plots.
One can observe that the LSNS and LSL are almost indistinguishable and, as predicted by the theory, perform better
than the LSS and the greedy algorithm.\label{fig3:2}
\vspace{-5pt}}
\end{figure}

\paragraph{Homoscedastic noise}

We chose $n=d=200$ and randomly generated a $n\times d$ matrix $\btheta$
with i.i.d.\ entries uniformly distributed on $[0,\tau]$, with several values
of $\tau$ varying between $1.4$ and  $3.5$. Then, we randomly chose a permutation
$\pi^*$ (uniformly from $\Sn$) and generated the sets $\{X_i\}$ and $\{\Xdiese_i\}$
according to (\ref{model}) with $\sigma_i=\sdiese_i=1$. Using these sets as data,
we computed the four estimators of $\pi^*$ and evaluated the average error rate
$\frac1n\sum_{i=1}^n \fcar_{\{\hat\pi(i)\not=\pi^*(i)\}}$. The result, averaged over
500 independent trials, is plotted in Fig.\ \ref{fig3:1}.

Note that the three estimators originating from the maximum likelihood methodology lead to the same estimators, while the greedy algorithm
provides an estimator which is much worse than the others when the parameter $\k$ is small.

\paragraph{Heteroscedastic noise}

This experiment is similar to the previous one, but the noise level is not constant.
We still chose $n=d=200$ and defined $\btheta= \tau I_d$, where $I_d$ is
the identity matrix and $\tau$ varies between $4$ and  $10$. Then, we randomly chose a permutation
$\pi^*$ (uniformly from $\Sn$) and generated the sets $\{X_i\}$ and $\{\Xdiese_i\}$
according to (\ref{model}) with $\sigma_{\pi^*(i)}=\sdiese_i=1$ for 10 randomly chosen values of $i$ and
$\sigma_{\pi^*(i)}=\sdiese_i=0.5$ for the others. Using these sets as data,
we computed the four estimators of $\pi^*$ and evaluated the average error rate
$\frac1n\sum_{i=1}^n \fcar_{\{\hat\pi(i)\not=\pi^*(i)\}}$. The result, averaged over
500 independent trials, is plotted in Fig.\ \ref{fig3:2}.

Note that among the noise-level-adaptive estimators, LSL outperforms the two others and
is as accurate as, and even slightly better than the LSNS pseudo-estimator. This confirms
the theoretical findings presented in foregoing sections.

\section{Conclusion and Future Work}

Motivated by the problem of feature matching, we proposed a rigorous framework for studying the
problem of permutation estimation from a minimax point of view. The key notion in our framework
is the minimax rate of separation, which plays the same role as in the statistical hypotheses testing
theory \citep{Ingster}. We established theoretical guarantees for several natural estimators and proved
the optimality of some of them. The results appeared to be quite different in the homoscedastic and in
the heteroscedastic cases. However, we have shown that the least sum of logarithms estimator outperforms
the other procedures both theoretically and empirically.

Several avenues of future work have been already mentioned in previous sections. In particular, investigating
the statistical properties of the arrangement estimation problem described in Section~\ref{sec:arrangement}
and considering the case of unspecified transformation relating the features may have a significant impact on
the practice of  feature matching.

Another interesting question is to extend the statistical inference developed here for the problem of feature
matching to the more general assignment problem. The latter aims at assigning $m$ tasks to $n$ agents such that
the cost of assignment is as small as possible. Various settings of this problem have been considered in the
literature \citep{Pentico}  and many algorithms for solving the problem have been proposed \citep{Romeijn}. However,
to the best of our knowledge, the statistical aspects of the problem in the case where the cost matrix is
corrupted by noise have not been studied so far.

\section{Proofs of the Theorems}\label{section:theorems}

In this section we collect the proofs of the theorems. We start with the proof of Theorem~\ref{upperbound2}, since it concerns the more general setting and the proof of Theorem~\ref{upperbound1} can be deduced from that of Theorem~\ref{upperbound2} by simple arguments. We then prove the other theorems in the usual order and postpone the proofs of some technical lemmas to the next section.

\begin{proofof}[Proof of Theorem~\ref{upperbound2}]
To ease notation and without loss of generality, we assume that $\pi^*$ is the identity permutation denoted by $\id$. Furthermore, since there is no risk of confusion, we write $\prob$ instead of $\prob_{\btheta,\bsigma,\pi^*}$. We wish to bound the probability of the event $\O = \{\hat\pi\neq \id\}$.

Let us first denote by $\hat\pi$ the maximum likelihood estimator $\pi^{\textrm{LSL}}$ defined by (\ref{LSL}). We have
\begin{equation*}
\O \subset \bigcup_{\pi\not=\id} \O_\pi,
\end{equation*}
where
\begin{align*}
\O_\pi &
		= \Big\{\sum_{i=1}^n \log \frac{\|X_i-\Xdiese_i\|^2}{\|X_{\pi(i)}-\Xdiese_i\|^2}\ge 0\Big\}
		= \Big\{\sum_{i : \pi(i)\not= i} \log \frac{\|X_i-\Xdiese_i\|^2}{\|X_{\pi(i)}-\Xdiese_i\|^2}\ge 0\Big\}.
\end{align*}
On the one hand, for every permutation $\pi$,
\begin{align*}
\sum_{\pi(i)\not =i} \log\Big({\textstyle\frac{2\s_i^2}{\s_i^2+\s_{\pi(i)}^2}}\Big)
 &= \sum_{i=1}^n \big(\log(2\s_i^2) - \log(\s_i^2+\s_{\pi(i)}^2) \big)\\
 &=  \sum_{i=1}^n \frac{\log(2\s_i^2)+\log(2\s_{\pi(i)}^2)}2-\log(\s_i^2+\s_{\pi(i)}^2)
\end{align*}
so, using the concavity of the logarithm, this quantity is nonpositive. Therefore,
\begin{align*}
\O_\pi
    &\subset  \Big\{\sum_{i : \pi(i)\not= i}\!\! \log \textstyle\frac{\|X_i-\Xdiese_i\|^2/(2\s_i^2)}
        {\|X_{\pi(i)}-\Xdiese_i\|^2/(\s_i^2+\s_{\pi(i)}^2)}\ge 0\Big\}\\
    &\subset \bigcup_{i=1}^n \bigcup_{j\not =i}\Big\{\frac{\|X_i-\Xdiese_i\|^2}{2\s_i^2}\ge\frac{\|X_{j}-\Xdiese_i\|^2}{\s_j^2+\s_{i}^2}\Big\}.
\end{align*}
This readily yields $\O\subset  \bar\O$, where
\begin{align}\label{inclusion}
\bar\O =\bigcup_{i=1}^n \bigcup_{j\not =i}\Big\{\frac{\|X_i-\Xdiese_i\|^2}{2\s_i^2}\ge \frac{\|X_{j}-\Xdiese_i\|^2}{\s_j^2+\s_{i}^2}\Big\}.
\end{align}
Furthermore, the same inclusion is true for the LSNS estimator as well. Therefore, the rest of the proof
is common for the estimators LSNS and LSL.

We set $\sigma_{i,j}=(\sigma_i^2+\sigma_j^2)^{1/2}$ and
\begin{align*}
\z_1 &= \max_{i\not =j} \bigg|
\frac{(\t_{i}-\t_j)^\top(\s_i\xi_i-\s_j\xidiese_j)}{\|\t_i-\t_j\|\s_{i,j}}\bigg|, \quad
\z_2 = d^{-1/2} \max_{i,j }  \bigg|\,\bigg\|\frac{\s_i\xi_i-\s_j\xidiese_j}{\s_{i,j}}\bigg\|^2 - d\bigg|.
\end{align*}
Since $\pi^*=\id$, it holds that for every $i\in\{1,\ldots,n\}$,
\begin{align*}
\|X_i-\Xdiese_i\|^2 &=  \s_i^2 \|\xi_i-\xidiese_i\|^2 \le 2\s_i^2(d+\sqrt{d}\z_2).
\end{align*}
Similarly, for every $j\not =i$,
\begin{align*}
\|X_{j}-\Xdiese_i\|^2
    &=  \|\t_{j}-\t_i\|^2+ \|\s_{j}\xi_{j}-\s_i\xidiese_i\|^2+ 2(\t_{j}-\t_i)^\top(\s_{j}\xi_{j}-\s_i\xidiese_i).
\end{align*}
Therefore,
\begin{align*}
\|X_{j}-\Xdiese_i\|^2
    &\ge  \|\t_{j}-\t_i\|^2+\s_{i,j}^2 (d-\sqrt{d}\z_2) - 2\|\t_{j}-\t_i\|\s_{i,j}\z_1.
\end{align*}
This implies that on the event $\O_1 = \{\bar\k(\btheta,\bsigma)\ge \z_1\}$  it holds that
\begin{align*}
\frac{\|X_{j}-\Xdiese_i\|^2}{\s_{i,j}^2}
    \ge  \bar\k(\btheta,\bsigma)^2 - 2\bar\k(\btheta,\bsigma)\z_1+d-\sqrt{d}\z_2.
\end{align*}
Combining these bounds, we get that
$$
\O\cap \O_1\subset  \Big\{d+\sqrt{d}\z_2\ge\bar\k(\btheta,\bsigma)^2 - 2\bar\k(\btheta,\bsigma)\z_1+d-\sqrt{d}\z_2\Big\},
$$
which implies that
\begin{align}
\prob(\O)
        &\le \prob(\O_1^\complement) + \prob\big(\O\cap\O_1\big)\nonumber\\
        &\le \prob\big(\z_1\ge \bar\k(\btheta,\bsigma)\big)
            +\prob(2\sqrt{d}\z_2+2\bar\k(\btheta,\bsigma)\z_1\ge\bar\k(\btheta,\bsigma)^2)\nonumber\\
        &\le 2\prob\big(\z_1\ge \textstyle\frac{\bar\k(\btheta,\bsigma)}4\big)+
            \prob\big(\z_2\ge \textstyle\frac{\bar\k(\btheta,\bsigma)^2}{4\sqrt{d}}\big).\label{ineq:2}
\end{align}
Finally, one easily checks that for suitably chosen random variables $\z_{i,j}$ drawn from the standard Gaussian distribution,
it holds that $\z_1 = \max_{i\not=j}|\z_{i,j}|$. Therefore, using the well-known tail bound for the standard Gaussian
distribution in conjunction with the union bound, we get
\begin{align}
\prob\big(\z_1\ge \textstyle\frac14\bar\k(\btheta,\bsigma)\big)
    &\le \sum_{i\not =j}\nolimits \prob\big(|\z_{i,j}|\ge \textstyle\frac14\bar\k(\btheta,\bsigma)\big)
    \le 2 n^2 e^{-\frac1{32}\bar\k(\btheta,\bsigma)^2}.\label{ineq:3}
\end{align}
To bound the large deviations of the random variable $\zeta_2$, we rely on the following result.
\begin{lemma}[\cite{LaurentMassart2000}, Eq. (4.3) and (4.4)]\label{concentration}
If $Y$ is drawn from the chi-squared distribution  $\chi^2(D)$, where $D\in\mathbb N^*$, then, for every $x>0$,
\begin{equation*}
\begin{cases}
\ \prob\big(Y-D \le -2\sqrt{Dx} \big) \le e^{-x}, \phantom{\Big()}\\
\ \prob\big(Y-D \ge 2\sqrt{Dx} + 2x \big) \le e^{-x}. \phantom{\Big()}
\end{cases}
\end{equation*}
As a consequence, $\forall y>0$,
$\prob\big(D^{-1/2}|Y-D| \ge y\big) \le 2\exp\big\{-\textstyle\frac18 y(y\wedge \sqrt{D})\big\}$.
\end{lemma}
This inequality, combined with the union bound, yields
\begin{align}
\prob\Big(\z_2 \ge \textstyle\frac{\bar\k(\btheta,\bsigma)^2}{4\sqrt{d}}\Big)
        &\le 2n^2 \exp\Big\{-\frac{(\bar\k(\btheta,\bsigma)/16)^2}{d}(\bar\k^2(\btheta,\bsigma)\wedge 8d)\Big\}.\label{ineq:4}
\end{align}
Combining inequalities (\ref{ineq:2})-(\ref{ineq:4}), we obtain that as soon as
\begin{align*}
\bar\k(\btheta,\bsigma)\ge 4\Big(\sqrt{2\log(8n^2/\a)} \vee \big(d\log(4n^2/\a)\big)^{1/4}\Big),
\end{align*}
we have $\prob(\hat\pi\not = \pi^*)=\prob(\O)\le \a$.
\vspace{-20pt}

~
\end{proofof}

\begin{proofof}[Proof of Theorem~\ref{upperbound1}]
Without loss of generality, we assume $\pi =\id$. It holds that, on the event
$$
\mathcal A= \bigcap_{i=1}^n\bigcap_{j\not =\pi^*(i)}\Big\{ \|X_{\pi^*(i)}-\Xdiese_{i}\|< \|X_{j}-\Xdiese_{i}\|\Big\},
$$
all the four estimators coincide with the true permutation $\pi^*$. Therefore, we have
$$
\{\hat \pi\not=\pi^*\} \subseteq \bigcup_{i=1}^n\bigcup_{j\not =i}\Big\{ \|X_{i}-\Xdiese_{i}\|\geq  \|X_{j}-\Xdiese_{i}\|\Big\}.
$$
The latter event is included in $\bar\O$ at the right-hand side of (\ref{inclusion}), the probability of which has been
already shown to be small in the previous proof.
\end{proofof}

\begin{proofof}[Proof of Theorem~\ref{lowerbound1}]
We refer the reader to the proof of Theorem~\ref{lowerbound2} below, which concerns the more general situation. Indeed,
when all the variances $\s_j$ are equal, Theorem~\ref{lowerbound2} boils down to
Theorem~\ref{lowerbound1}.
\end{proofof}

\begin{proofof}[Proof of Theorem~\ref{lowerbound2}]
To establish lower bounds for various types of risks we will  use the following lemma:
\begin{lemma}[\citet{Tsybakov2009}, Theorem 2.5]\label{Tsybakov}
Assume that for some integer $M\ge 2$ there exist distinct permutations $\pi_0,\ldots,\pi_M\in\Sn$ and mutually absolutely continuous
probability measures $\bQ_0,\ldots,\bQ_M$ defined on a common probability space $(\mathcal Z,\mathscr Z)$ such that
\begin{equation*}
 \frac{1}{M} \sum_{j=1}^M K(\bQ_j,\bQ_0) \le \frac{1}{8} \log M.
\end{equation*}
Then, for every measurable mapping $\tilde\pi:\mathcal Z\to\Sn$,
\begin{equation*}
\max_{j=0,\ldots,M} \bQ_j(\tilde\pi\neq\pi_j) \ge \frac{\sqrt{M}}{\sqrt{M}+1} \Big( \frac{3}{4}-\frac{1}{2\sqrt{\log M}} \Big).
\end{equation*}
\end{lemma}
To prove Theorem~\ref{lowerbound2}, we consider separately the two cases
\begin{align*}
&\textbf{Case 1:}\qquad \max\Big\{(\log n)^{1/2} , (c d \log n)^{1/4}\Big\} = (\log n)^{1/2},\\
&\textbf{Case 2:}\qquad \max\Big\{(\log n)^{1/2} , (c d \log n)^{1/4} \Big\} = (c d \log n)^{1/4}.
\end{align*}
\paragraph{\textbf{Case 1:} We assume that $\k \le \frac18 \sqrt{\log n}$ }
Denote $m$ the largest integer such that $2m\le n$. We assume without loss of generality that the noise levels are ranked in increasing order: $\s_1\leq\ldots\leq\s_n.$
Then, we construct a least favorable set of vectors for the estimation of the permutation. To ease notation, we set $\s_{i,j}=(\s_i^2+\s_j^2)^{1/2}$.
\begin{lemma}\label{construction_btheta}
Assume that $m$ is the largest integer such that $2m\le n$. Then there is a set of vectors $\btheta$ such that
\begin{equation*}
\frac{\|\t_1-\t_2\|}{\s_{1,2}} = \ldots = \frac{\|\t_{2m-1}-\t_{2m}\|}{\s_{2m-1,2m}} = \k,
\end{equation*}
and for every pair $\{i,j\}$ different from the pairs $\{1,2\},\ldots,\{2m-1,2m\} $ we have
\begin{equation*}
\frac{\|\t_i-\t_j\|}{\s_{i,j}} > \k\bigg(1+\frac{\max_{1\le\ell\le n} \sigma_\ell}{\min_{1\le\ell\le n} \sigma_\ell}\bigg).
\end{equation*}
\end{lemma}

Denote $\btheta^0$ the constructed set. We define for every $k\in\{1,\ldots,m\}$, the set $\btheta^k=\btheta^0 + (0,\ldots,\eta_k,\eta_k,\ldots,0),$
where only the $(2k-1)^\text{th}$ and $2k^\text{th}$ components are modified by adding the vector
$\eta_k=(\s_{2k}^2-\s_{2k-1}^2)(\t_{2k-1}-\t_{2k})/(2\s_{2k}^2+2\s_{2k-1}^2)$.
It follows from Lemma~\ref{construction_btheta} that $\btheta^0,\ldots,\btheta^m \text{ belong to } \bar\T_\k,$ so that, denoting for every $k\in\{1,\ldots,m\}, \pi_k=(2k-1~2k)$ the transposition of $\Sn$ that only permutes $2k-1$ and $2k$, and $\pi_0=\id,$ we get the following lower bound for the risk:
\begin{equation*}
\inf_{\hat\pi} \sup_{(\pi,\btheta)\in\Sn\times\bar\T_\k} \prob_{\btheta,\pi} (\hat\pi\neq\pi) \ge \inf_{\hat\pi} \max_{k=0,\ldots,m} \prob_{\btheta^{k},\pi_k}(\hat\pi\neq\pi_k).
\end{equation*}
In order to use Lemma~\ref{Tsybakov} with $\bQ_j=\prob_{\btheta^j,\pi_j}$, we compute $\forall k\in\{1,\ldots,m\}$
\begin{align*}
2K(\prob_{\btheta^k,\pi_k},\prob_{\btheta^0,\pi_0}) = &\frac{\|\eta_k\|^2}{2\s_{2k-1}^2}+
\frac{\|\t_{2k-1}+\eta_k-\t_{2k}\|^2}{\s^2_{2k}} + d\Big(\frac{\s^2_{2k-1}}{\s^2_{2k}}-1\Big) \\
& +\frac{\|\eta_k\|^2}{2\s_{2k}^2}+\frac{\|\t_{2k}+\eta_k-\t_{2k-1}\|^2}{\s^2_{2k-1}} + d\Big(\frac{\s^2_{2k}}{\s^2_{2k-1}}-1\Big).
\end{align*}
Using the definition of $\eta_k$, we get
\begin{align*}
2K(\prob_{\btheta^k,\pi_k},\prob_{\btheta^0,\pi_0})
&=  \frac{\|\t_{2k}-\t_{2k-1}\|^2}{\s_{2k}^2+\s_{2k-1}^2}
\bigg(\frac{3\s_{2k-1}^2}{8\s_{2k}^2}+\frac{3\s_{2k}^2}{8\s_{2k-1}^2}+\frac{10}{8} \bigg)
+ d \frac{\s^2_{2k-1}}{\s^2_{2k}} \Big( 1 - \frac{\s^2_{2k}}{\s^2_{2k-1}} \Big)^2 \\
&\le \k^2\bigg(\frac{13}{8}+\frac{3\s_{2k}^2}{8\s_{2k-1}^2}\bigg) + \frac{\log n}{16}.
\end{align*}
Next, we apply the following result.
\begin{lemma}
Let $a_1,a_2,\ldots,a_m$ be real numbers larger than one such that $\prod_{k=1}^m a_k\le A$. Then,
$\sum_{k=1}^m a_k \le n+\log A\max_k a_k$.
\end{lemma}
\begin{proof}
We use the simple inequality $e^x\le 1+xe^x$ for all $x\ge 0$. Replacing $x$ by $\log a_k$ and summing over
$k=1,\ldots,m$ we get
$$
\sum_{k=1}^m a_k \le \sum_{k=1}^m 1+ a_k\log a_k\le m+ \max_{k=1,\ldots,m} a_k \log \prod_{k=1}^m a_k \le m+ \max_{k=1,\ldots,m} a_k \log A.
$$
This completes the proof of the lemma.
\end{proof}
We apply this lemma to $a_k = \s_{2k}^2/\s_{2k-1}^2$. Since the variances are sorted in increasing order, we have
$\prod_{k=1}^m a_k \le \prod_{i=1}^{n-1} {\s_{i+1}^2}/{\s_i^2} = \s_n^2/\s_1^2\le 1+\frac14(\frac{\log n}{d})^{1/2}$. In conjunction
with the inequality $\log(1+x)\le x$, this entails that $\sum_{k=1}^m \s_{2k}^2/\s_{2k-1}^2\le m+ \frac14(1+\frac14 (\frac{\log n}{d})^{1/2})(\frac{\log n}{d})^{1/2}$. Then, since $\log n\le 2\log m$ for $n\ge 6$ and $ (\frac{\log n}{d})^{1/2}\le  ({\log n})^{1/2}\le \frac34\log n$, we get
\begin{align*}
\frac1m\sum_{k=1}^m K(\prob_{\btheta^k,\pi_k},\prob_{\btheta^0,\pi_0})
&\le\k^2\bigg(\frac{13}{16}+\frac{3}{16m}\sum_{k=1}^m\frac{\s_{2k}^2}{\s_{2k-1}^2}\bigg) + \frac{\log m}{16}\\
&\le\k^2\bigg(1+\frac{3}{64m}\Big\{1+\frac14 \Big(\frac{\log n}{d}\Big)^{1/2}\Big\}\Big(\frac{\log n}{d}\Big)^{1/2}\bigg) + \frac{\log m}{16}\\
&\le\k^2\Big(1+\frac{3}{64m}\Big) + \frac{\log m}{16}.
\end{align*}
Finally, using the fact that $m\ge 3$, we get $\frac1m\sum_{k=1}^m K(\prob_{\btheta^k,\pi_k},\prob_{\btheta^0,\pi_0}) \le
\frac{65}{64}\k^2+\frac{\log m}{16} \le  \frac{\log m}{8}$ since $\k^2\le \frac2{65}\log n \le \frac4{65}\log m$.
We conclude by Lemma~\ref{Tsybakov} and by the monotonicity of the function $m\mapsto \frac{\sqrt{m}}{1+\sqrt{m}}(\frac34-\frac1{2\sqrt{\log m}})$ that
\begin{equation*}
\inf_{\hat\pi} \sup_{(\pi,\btheta)\in\Sn\times\bar\T_\k} \prob_{\btheta,\pi} (\hat\pi\neq\pi) \ge \frac{\sqrt{3}}{\sqrt{3}+1} \Big( \frac{3}{4} - \frac{1}{2\sqrt{\log3}} \Big)\ge 0.17.
\end{equation*}

\paragraph{\textbf{Case 2:} We assume that $\frac18 \sqrt{\log n}\le \k \le \frac{1}{8} (cd\log n)^{1/4}$}
In this case, we have $d\ge\frac{1}{c}\log n$.  To get the desired result, we use Lemma~\ref{Tsybakov} for a
properly chosen family of probability measures described below.

\begin{lemma}\label{kullback_leibler_divergence_heteroscedastic}
Let $\e_1,\ldots,\e_n$ be real numbers defined by
$$
\e_{{k}} = \sqrt{2/d}\;\k\s_{{k}},\qquad \forall{k}\in\{1,\ldots,n\},
$$
and let $\mu$ be the uniform distribution on $\mathcal E=\{\pm\e_1\}^d\times\ldots\times\{\pm\e_{n}\}^d$. We denote by $\prob_{\mu,\pi}$ 
the probability measure on $\RR^{d\times n}$ defined by $\prob_{\mu,\pi}(A) = \int_{\mathcal E} \prob_{\btheta,\pi}(A)\,\mu(d\btheta)$. 
Assume that $\s_1\le\ldots\le \s_n$. For two positive integers $k<k'\le n$, set $\gamma = \frac{\s_{k'}^2}{\s_{k}^2}$ and let $\pi=({k}~{k}')$ 
be the transposition that only permutes ${k}$ and ${k}'$. Then
\begin{align*}
K(\prob_{\mu,\pi},\prob_{\mu,\id}) &\le  {4\k^2}\big(1-\gamma^{-1}\big)+\frac{8\k^4}{d}\big(2+\big(1+(2/d)\k^2\big)^2\gamma^{2}\big)
+ \frac12\big(d+2\k^2\big)(\gamma-1)^2
\end{align*}
and $\mu(\mathcal E\backslash\bar\T_\k) \le ({n(n-1)}/{2})\ e^{-d/8}$.
\end{lemma}

The assumption on the noise levels entails that, for any integer $k\in\{1,\ldots,k'\}$,
$1\le \frac{\s_{k'}^2}{\s_{k}^2} \le 1+\frac14(\frac{\log n}{d})^{1/2}$,
and consequently, $(\gamma-1)^2=\big(\frac{\s_{k'}^2}{\s_{k}^2} -1\big)^2\le 4^{-2}\big(\frac{\log n}{d}\big)$.
Furthermore, $\frac{\k^2}{d}\le \frac{c}{64}\le \frac1{64}$ provided that $c\le 1$.
Finally, for the Kullback-Leibler divergence between $\prob_{\mu,\pi_{k,k'}}$ and $\prob_{\mu,\id}$, where
$\pi_{k,k'}=(k~k')$ is the transposition from $\Sn$ permuting only $k$ and $k'$, it holds
\begin{align*}
K(\prob_{\mu,\pi_{k,k'}},\prob_{\mu,\id}) &\le {\k^2}\sqrt{\frac{\log n}{d}}+\frac{8\k^4}{d}\Big(2+\frac{33^2}{32^2}\big(1+0.25\big)^2\Big)+
\frac{33d}{64}\times\frac{\log n}{16d}\\
&\le \frac{\log n}{8}\le \frac{\log n(n-1)/2}{8},
\end{align*}
where we have used once again the facts that $c\le 1$ and $n\ge 3$.
Applying Lemma~\ref{Tsybakov} with
$M = n(n-1)/2$, $\bQ_0 = \prob_{\mu,\id}$ and $\{\bQ_j\}_{j=1,\ldots,M}=\{\prob_{\mu,\pi_{k,k'}}\}_{k\not= k'}$,
we obtain
\begin{align*}
\max_{\pi^*\in\Sn}\sup_{\btheta\in\bar\T_\k}\prob_{\btheta,\pi^*}\big(\hat\pi\neq\pi^*\big) &\ge \max_{\pi^*\in\{\id\}\cup\{\pi_{k,k'}\}}
\int_{\bar\T_\k} \prob_{\btheta,\pi^*} \big( \hat\pi \neq \pi^* \big) \frac{\mu(d\btheta)}{\mu(\bar\T_\k)} \\
&\ge \max_{\pi^*\in\{\id\}\cup\{\pi_{k,k'}\}} \prob_{\mu,\pi^*} \big( \hat\pi\neq\pi^*\big) - \mu(\mathcal E\backslash\bar\T_\k) \\
&\ge \frac{\sqrt{15}}{\sqrt{15}+1} \Big( \frac{3}{4} - \frac{1}{2\sqrt{\log15}} \Big) - \frac{n(n-1)}2 e^{-d/8}.
\end{align*}
In view of the inequalities $d\ge (1/c)\log n$, $c\le 1/20$ and $n\ge 6$, we get
the inequality $\max_{\pi^*\in\Sn}\sup_{\btheta\in\bar\T_\k}\prob_{\btheta,\pi^*}\big(\hat\pi\neq\pi^*\big)\ge 22.4\%$.
\end{proofof}

\begin{proofof}[Proof of Theorem~\ref{lowerbound3}]
Since the event $\{\pi^{\textrm{gr}}\not =\id\}$  includes the event
$$
\O_2=\{\|X_{1}-\Xdiese_1\|^2>\|X_{2}-\Xdiese_1\|^2\},
$$
it is sufficient to bound from below the probability of $\O_2$.
To this end, we choose any $\btheta\in\RR^{n\times d}$ satisfying $\|\t_1-\t_2\|=2\k$. This readily implies that
$\btheta$ belongs to $\bar\T_\k$.  Furthermore, for suitably chosen random variables $\eta_1\sim \chi^2_d$, 
$\eta_2\sim \chi^2_d$ and $\eta_3\sim\mathcal N(0,1)$, it holds that
$\|X_{1}-\Xdiese_1\|^2-\|X_{2}-\Xdiese_1\|^2 = 6\eta_1 - 4\k^2-8\k \eta_3 - 4\eta_2$.
The random terms in the last sum can be controlled using Lemma~\ref{concentration}.
More precisely, for every $x>0$, each one of the following three inequalities holds true
with probability at least $1-e^{-x^2}$:
\begin{align*}
\eta_1 \ge  d- 2\sqrt{d} x,\quad
\eta_2 \le  d+ 2\sqrt{d} x + 2x^2,\quad
\eta_3 \le \sqrt{2} x.
\end{align*}
This implies that with probability at least $1-3e^{-x^2}$, we have
\begin{equation*}
\|X_{1}-\Xdiese_1\|^2-\|X_{2}-\Xdiese_1\|^2 \ge  2d-20\sqrt{d} x-4(\k + \sqrt2 x)^2.
\end{equation*}
If $x=\sqrt{\log 6}$, then the conditions imposed on $\k$  and $d$ ensure that the right-hand side of the last inequality is positive. Therefore,
$\prob(\O_2)\ge 1-3e^{-x^2}=1/2$.
\end{proofof}

\begin{proofof}[Proof of Theorem~\ref{lowerbound4}]
The proof is split into two parts. In the first part, we consider the case
$\k\le \frac14\sqrt{\frac{\log M_n}{n}}$, while in the second part the case $\kappa\le \frac18\big(\frac{\log n}{d}\big)^{1/4}$ with
$d\ge 24\log n$ and is analyzed. In both cases, the main tool we use is the following result.

\begin{lemma}[\citet{Tsybakov2009}, Theorem 2.5]\label{Tsybakov2}
Assume that for some integer $M\ge 2$ there exist distinct permutations $\pi_0,\ldots,\pi_M\in\Sn$ and mutually absolutely continuous
probability measures $\bQ_0,\ldots,\bQ_M$ defined on a common probability space $(\mathcal Z,\mathscr Z)$ such that
\begin{equation*}
\begin{cases}
\ \exists\ s>0,\ \forall\ i\neq j, \ \d(\pi_i,\pi_j)\ge 2s, \\
\  \frac{1}{M} \sum_{j=1}^M K(\bQ_j,\bQ_0) \le \frac{1}{8} \log M.
\end{cases}
\end{equation*}
Then, for every measurable mapping $\tilde\pi:\mathcal Z\to\Sn$,
\begin{equation*}
\max_{j=0,\ldots,M} \bQ_j\big(\d(\hat\pi,\pi_j)\ge s\big) \ge \frac{\sqrt{M}}{\sqrt{M}+1} \Big( \frac{3}{4}-\frac{1}{2\sqrt{\log M}} \Big).
\end{equation*}
\end{lemma}

We now have to choose $M$ and $\pi_0,\ldots,\pi_M$ in a suitable manner, which will be done differently according to the relationship between
$n$ and $d$.

\paragraph{\textbf{Case 1:} We assume that $\k\le \frac14\sqrt{\frac{\log M_n}{n}}$}
Let $M = M_n$ with $M_n= \mathcal M(1/4, B_{2,n}(2),\delta_H)$ and let $\btheta = (\theta_1,\ldots,\theta_n)$ be the set of vectors
$\theta_k = k\k\s(1,0,\ldots,0)\in\RR^d$. Clearly, $\btheta$ belongs to $\bar\T_\k$. By definition of the packing number, there exist $\pi_1,\ldots,\pi_{M_n}$, permutations from $\Sn$,
such that
$$
\delta_2(\pi_j,\id)\le 2,\qquad \delta_H(\pi_i,\pi_j)\ge \frac14;\qquad\forall i,j\in\{1,\ldots,M_n\}, \ i\not=j.
$$
Defining $\bQ_j=\prob_{\btheta,\pi_j}$ for $j=1,\ldots,M_n$ and $\bQ_0=\prob_{\btheta,\id}$ we get
\begin{align*}
K(\bQ_j,\bQ_0)
	&= \frac1{2\s^2} \sum_{k=1}^n \|\theta_{\pi_j(k)}-\theta_{k}\|^2 = \frac{\k^2}{2}\sum_{k=1}^n (\pi_j(k)-k)^2\\
	&= \frac{n\k^2}{2}\;\delta_2(\pi_j,\id)^2\le 2n\k^2.
\end{align*}
Therefore, using Lemma~\ref{Tsybakov2} with $s=1/8$ we infer from $\k\le \frac14\big(\frac{\log M_n}{n}\big)^{1/2}$ that
$$
\min_{\hat\pi}\max_{j=0,\ldots,M_n} \prob_{\btheta,\pi_j}(\delta_H(\hat\pi,\pi_j)\ge 1/8) \ge \frac{\sqrt{3}}{\sqrt{3}+1}\Big(\frac34-\frac1{2\sqrt{\log 3}}\Big) \approx 17.31\%.
$$
As a consequence, we obtain that
\begin{align*}
\min_{\hat\pi}\max_{(\pi,\btheta)\in\Sn\times\bar\T_\k} \esp_{\btheta,\pi}[\delta_H(\hat\pi,\pi)]
    &\ge \min_{\hat\pi}\max_{j=0,\ldots,M} \esp_{\btheta,\pi}[\delta_H(\hat\pi,\pi)\fcar_{\{\delta_H(\hat\pi,\pi)\ge 1/8\}}]\\
    &\ge \frac18 \min_{\hat\pi}\max_{j=0,\ldots,M} \esp_{\btheta,\pi}[\fcar_{\{\delta_H(\hat\pi,\pi)\ge 1/8\}}]\\
    &\ge 2.15\%.
\end{align*}
This completes the proof of the first case.

\paragraph{\textbf{Case 2:} We assume that $d\ge 24\log n$ and $\k \le \frac18 (d\log n)^{1/4}$ }
Let $\mu$ be the uniform distribution on $ \{\pm\e\}^{m\times d}$ with $\e={\sqrt{2/d}\;\s\k}$, as in
Lemma~\ref{kullback_leibler_divergence_heteroscedastic}. For any set of permutations $\{\pi_0,\ldots,\pi_M\}\subset\Sn$,
in view of Markov's inequality,
\begin{align*}
\sup_{(\pi,\btheta)\in\Sn\times\bar\T_\k} \esp_{\btheta,\pi} \big[\delta_H(\hat\pi,\pi)\big] &\ge
\frac{3}{16} \Big( \max_{i=0,\ldots,M} \prob_{\mu,\pi_i} \Big(\delta_H(\hat\pi,\pi_i) \ge \frac{3}{16} \Big) - \mu(\bar\T_\k^\complement)\Big).
\end{align*}
We choose $M$ and $\pi_0,\ldots,\pi_M$ as in the following lemma.
\begin{lemma}\label{existence2}
For any integer $n\ge 4$ there exist permutations $\pi_0,\ldots,\pi_M\in\Sn$ such that
$$
\pi_0=\id,\qquad\ M\ge(n/24)^{n/6},
$$
each $\pi_i$  is a composition of at most $n/2$ transpositions with disjoint supports,
and for every distinct pair of indices $i,j\in\{0,\ldots,M\}$ we have
\begin{equation*}
\delta_H(\pi_i, \pi_j) \ge {3}/{8}.
\end{equation*}
\end{lemma}
As $\pi_i$ is a product of transpositions, the Kullback-Leibler divergence between $\prob_{\mu,\pi_i}$ and $\prob_{\mu,\pi_0}$ can be computed by independence thanks to Lemma~\ref{kullback_leibler_divergence_heteroscedastic}:
\begin{equation*}
\frac1M\sum_{i=1}^M K(\prob_{\mu,\pi_i},\prob_{\mu,\pi_0}) \le \frac{n}{2} \times  \frac{8\k^4}{d}\Big(2+\Big[1+\frac{2\k^2}{d}\Big]^2\Big) = \frac{16n\k^4}{d},
\end{equation*}
where the last inequality follows from the bound $\k\le 0.45 d^{1/2}$. For $n\ge 26$, it holds that $M\ge 2$ and
$$
\log M\ge \frac{n(\log n -\log 24)}{6}\ge \frac{\log n}{512}.
$$
Consequently, $\frac{16n\k^4}{d}\le \frac18\log M$ which allows us to apply Lemma~\ref{Tsybakov2}. This yields
\begin{align*}
\inf_{\hat\pi} \sup_{(\pi,\btheta)\in\Sn\times\bar\T_\k} \esp_{\btheta,\pi} [\delta_H(\hat\pi,\pi)]
&\ge \frac{3}{16}\Big[\frac{\sqrt{2}}{\sqrt{2}+1} \Big(\frac{3}{4}-\frac{1}{2\sqrt{\log 2}} \Big) - \frac{n^2}{2}e^{-d/8}\Big]\\
&\ge \frac{3}{16}\Big[0.077-\frac{n^2}{2}e^{-24\log(n)/8}\Big]\\
&\ge \frac{3}{16}\Big[0.077-\frac{1}{2n}\Big]\ge 5.81\%.
\end{align*}
\vspace{-28pt}

~
\end{proofof}


\section{Proofs of the Lemmas}\label{section:lemmas}

\begin{proofof}[Proof of Lemma~\ref{construction_btheta}]
Let us denote $r_\sigma =\max_{1\le\ell\le n} \sigma_\ell/\min_{1\le\ell\le n} \sigma_\ell$.  It suffices to set $\theta_1=0\in\RR^d$,
\begin{align*}
\theta_{2k+1}&=\kappa\Big(\s_{1,2}+\ldots+\s_{2k-1,2k}+k\big(1+r_\s\big),0,\ldots,0\Big)\in\RR^d,\\
\theta_{2k}&=\kappa\Big(\s_{1,2}+\ldots+\s_{2k-1,2k}+(k-1)\big(1+r_\s\big),0,\ldots,0\Big)\in\RR^d
\end{align*}
for all $k=1,\ldots,m-1$. If $n$ is impair, one can set $\theta_n=\theta_{n-1}+\k \big(1+r_\s\big) (1,0,\ldots,0)$.
One readily checks that these vectors satisfy the desired conditions.
\end{proofof}

\begin{proofof}[Proof of Lemma~\ref{kullback_leibler_divergence_heteroscedastic}]
Without loss of generality, we assume hereafter that $\pi\in\Sn$ is the transposition permuting $1$ and $2$.
Recall that the uniform distribution on $\{\pm\e_1\}^d\times\ldots\times\{\pm\e_{n}\}^d$ can also be written as the product
$\mu = \bigotimes_{\ell=1}^n\mu_\ell$, where $\mu_\ell$ is the uniform distribution on $\{\pm\e_{\ell}\}^d$.
Let us introduce an auxiliary probability distribution $\tilde\mu$ on $\RR^{n\times d}$ defined as $\tilde\mu = \delta_{\mathbf 0}\otimes\delta_{\mathbf 0}
\otimes \mu_2\otimes\ldots\otimes\mu_m$ with $\delta_{\mathbf 0}$ being the Dirac delta measure at $\mathbf 0\in\RR^d$.
 We set $\prob_{\tilde\mu,\id}(\cdot)= \int_{\Theta} \prob_{\btheta,\id}(\cdot)\tilde\mu(d\btheta)$.

We first compute the density of $\prob_{\mu,\pi}$ w.r.t.\ $\prob_{\mu,\id}$, which can be written as
\begin{align*}
\frac{d\prob_{\mu,\pi}}{d\prob_{\mu,\id}}(\bX,\bXdiese) = \frac{d\prob_{\mu,\pi}}{d\prob_{\tilde\mu,\id}}(\bX,\bXdiese) \bigg/ \bigg(\frac{d\prob_{\mu,\id}}{d\prob_{\tilde\mu,\id}}(\bX,\bXdiese)\bigg),\quad  \bX,\bXdiese\in \RR^{n\times d}.
\end{align*}
For every $\theta_i\in\RR^d$ we denote by $\prob_{\theta_i,\s_i}$ the probability distribution of $X_i$ from (\ref{model}), given by
\begin{equation*}
\frac{d\prob_{\t_i,\s_i}}{d\prob_{0,\s_j}}(x) = \exp\Big\{-\frac{\|\t_i\|^2}{2\s_i^2} + \frac{1}{\s_i^2}(x,\t_i)-\frac{\|x\|^2}{2}(\sigma_i^{-2}-\s_j^{-2})\Big\},\quad \forall x\in\RR^d.
\end{equation*}
With this notation, we have
\begin{align*}
&\frac{d\prob_{\mu,\pi}}{d\prob_{\tilde\mu,\id}}(\bX,\bXdiese) = \esp_{\mu} \Bigg[\frac{d\prob_{\t_1,\s_1}}{d\prob_{0,\s_1}}(X_1)\frac{d\prob_{\t_2,\s_2}}{d\prob_{0,\s_2}}(X_2)\frac{d\prob_{\t_2,\s_2}}{d\prob_{0,\s_1}}(\Xdiese_1)\frac{d\prob_{\t_1,\s_1}}{d\prob_{0,\s_2}}(\Xdiese_2)\Bigg] \\
&= \esp_{\mu} \Bigg[\frac{d\prob_{\t_1,\s_1}}{d\prob_{0,\s_1}}(X_1)\frac{d\prob_{\t_1,\s_1}}{d\prob_{0,\s_2}}(\Xdiese_2)\Bigg] \times \esp_{\mu}\Bigg[\frac{d\prob_{\t_2,\s_2}}{d\prob_{0,\s_2}}(X_2)\frac{d\prob_{\t_2,\s_2}}{d\prob_{0,\s_1}}(\Xdiese_1)\Bigg]\\
&= \prod_{k=1}^d\cosh\Big(\frac{\e_1}{\s_1^2}(X_{1,k}+\Xdiese_{2,k})\Big)\cosh\Big(\frac{\e_2}{\s_2^2}(X_{2,k}+\Xdiese_{1,k})\Big)\\
&\qquad\times\exp\Big\{-\frac12(\|\Xdiese_{1}\|^2-\|\Xdiese_{2}\|^2)(\s_2^{-2}-\s_1^{-2})\Big\}.
\end{align*}
Similarly,
\begin{align*}
\frac{d\prob_{\mu,\id}}{d\prob_{\tilde\mu,\id}}&(\bX,\bXdiese) =\esp_{\mu} \Bigg[\frac{d\prob_{\t_1,\s_1}}{d\prob_{0,\s_1}}(X_1)\frac{d\prob_{\t_2,\s_2}}{d\prob_{0,\s_2}}(X_2)\frac{d\prob_{\t_1,\s_1}}{d\prob_{0,\s_1}}(\Xdiese_1)
\frac{d\prob_{\t_2,\s_2}}{d\prob_{0,\s_2}}(\Xdiese_2)\Bigg] \\
&= \esp_{\mu} \Bigg[\frac{d\prob_{\t_1,\s_1}}{d\prob_{0,\s_1}}(X_1)\frac{d\prob_{\t_1,\s_1}}{d\prob_{0,\s_1}}(\Xdiese_1)\Bigg] \times \esp_{\mu}\Bigg[\frac{d\prob_{\t_2,\s_2}}{d\prob_{0,\s_2}}(X_2)\frac{d\prob_{\t_2,\s_2}}{d\prob_{0,\s_2}}(\Xdiese_2)\Bigg]\\
&= \prod_{k=1}^d\cosh\Big(\frac{\e_1}{\s_1^2}(X_{1,k}+\Xdiese_{1,k})\Big)\cosh\Big(\frac{\e_2}{\s_2^2}(X_{2,k}+\Xdiese_{2,k})\Big).
\end{align*}
Thus, we get that
\begin{align*}
\frac{d\prob_{\mu,\pi}}{d\prob_{\mu,\id}}(\bX,\bXdiese) &= \prod_{k=1}^d \frac{\cosh\big(\frac{\e_1}{\s_1^2}(X_{1,k}+\Xdiese_{2,k})\big)}{\cosh\big(\frac{\e_1}{\s^2_1}(X_{1,k}+\Xdiese_{1,k})\big)}
\times \prod_{k=1}^d \frac{\cosh\big(\frac{\e_2}{\s_2^2}(X_{2,k}+\Xdiese_{1,k})\big)}{\cosh\big(\frac{\e_2}{\s^2_2}(X_{2,k}+\Xdiese_{2,k})\big)}\\
&\qquad\times\exp\Big\{-\frac12(\|\Xdiese_{1}\|^2-\|\Xdiese_{2}\|^2)(\s_2^{-2}-\s_1^{-2})\Big\}.
\end{align*}
Then, we compute the Kullback-Leibler divergence,
\begin{align}
K&(\prob_{\mu,\pi},\prob_{\mu,\id})
	 = \int \log \Big(\frac{d\prob_{\mu,\pi}}{d\prob_{\mu,\id}}(\bX,\bXdiese)\Big) d\prob_{\mu,\pi}(\bX,\bXdiese)\nonumber\\
	& =  \sum_{k=1}^d\sum_{j=1}^2 \bigg\{\esp_{\mu} \bigg[ \int \log\cosh\Big[ \frac{\e_j}{\s^2_j}(2\t_{j,k}+\s_j\sqrt{2}x)\Big] \varphi(x)dx \bigg] \nonumber\\
	& \qquad -  \esp_{\mu}\bigg[ \int \log\cosh \Big[\frac{\e_j}{\s^2_j}(\t_{1,k}+\t_{2,k}+\s_{12}x)\Big]\varphi(x) dx \bigg]\bigg\}\nonumber\\
	&\qquad +\frac{d}2\esp_{\mu}\bigg[ \int_\RR ((\t_{1,1}+\s_1 x)^2-(\t_{2,1}+\s_2 x)^2)\varphi(x)\,dx\bigg](\s_2^{-2}-\s_1^{-2}),\nonumber
\end{align}
where $\varphi$ is the density function of the standard Gaussian distribution. We evaluate the first two terms of the last display using the following inequalities:
\begin{equation}\label{cosh}
\forall u\in\RR,\qquad
\frac{u^2}{2}-\frac{u^4}{12} \le \log\cosh(u) \le \frac{u^2}{2},
\end{equation}
while for the third term the exact computation yields:
\begin{align*}
\esp_{\mu}\bigg[ \int_\RR ((\t_{1,1}+\s_1 x)^2-(\t_{2,1}+\s_2 x)^2)\varphi(x)\,dx\bigg]&=
\e_1^2+\s_1^2-\e_2^2-\s_2^2\\
&= (\s_1^2-\s_2^2)(1+(2/d)\k^2).
\end{align*}
In conjunction with the facts that $\e_1/\s_1=\e_2/\s_2$, $\s_1\le\s_2$ and $\e_1\le\e_2$, this leads to
\begin{align*}
&\esp_{\mu} \bigg[ \int \log\cosh\Big[ \frac{\e_j}{\s_j^2}(2\t_{j,k}+\s_j\sqrt{2}x)\Big] \varphi(x)dx \le
\frac{\e_j^2}{\s_j^2} + 2\frac{\e_j^4}{\s_j^4}=\frac{\e_2^2}{\s_2^2} + 2\frac{\e_2^4}{\s_2^4},\\
&\esp_{\mu}\bigg[ \int \log\cosh \Big[\frac{\e_j}{\s_j^2}(\t_{1,k}+\t_{2,k}+\s_{1,2}x)\Big]\varphi(x) dx \bigg]\\
&\qquad\ge\frac{\e^2_j}{2\s^4_j}\Big(\e_1^2+\e_2^2+\s_1^2+\s_2^2\Big)\\
&\qquad\qquad- \frac{\e^4_j}{12\s^8_j}\Big(\e_1^4+\e_2^4+3(\s_1^2+\s_2^2)^2+6\e_1^2\e_2^2+6(\s_1^2+\s_2^2)(\e_1^2+\e_2^2)\Big)\\
&\qquad\ge\frac{\e^2_2(\e_1^2+\s_1^2)}{\s^4_2}- \frac{\e^4_1(\e_2^2+\s_2^2)^2}{\s^8_1}.
\end{align*}
Thus, we get that
\begin{align*}
(1/d)K(\prob_{\mu,\pi},\prob_{\mu,\id}) &\le \frac{2\e_2^2}{\s_2^2} + \frac{4\e_2^4}{\s_2^4}-\frac{2\e^2_2(\e_1^2+\s_1^2)}{\s^4_2}+\frac{2\e^4_1(\e_2^2+\s_2^2)^2}{\s^8_1}\\
&\qquad+ \frac12\big(1+(2/d)\k^2)(\s_1^2-\s_2^2)(\s_2^{-2}-\s_1^{-2})\\
&\le \frac{4\k^2}{d}\Big(1-\frac{\s_1^2}{\s_2^2}\Big)+\frac{16\k^4}{d^2}+\frac{8\k^4}{d^2}\big(1+(2/d)\k^2\big)^2\frac{\s_2^4}{\s_1^4}\\
&\qquad+ \frac12\big(1+(2/d)\k^2\big)\Big(\frac{\s_2^2}{\s_1^2}-1\Big)^2.
\end{align*}
To complete the proof, we need to evaluate $\mu(\mathcal E\setminus\bar\T_\k)$. We note that in view of the union bound,
\begin{align*}
\mu(\mathcal E\setminus\bar\T_\k) & = \mu\Big(\bigcup_{k=1}^n\bigcup_{k'\not=k}\{\btheta : \|\theta_{k}-\theta_{k'}\| < \k\s_{k,k'} \}\Big)\\
&\le \frac{n(n-1)}{2} \max_{k\not=k'}\mu\big(\{\btheta : \|\theta_{k}-\theta_{k'}\|^2 < \k^2\s_{k,k'}\}\big)\\
&= \frac{n(n-1)}{2}  \max_{k\not=k'}\prob\big(d\e_{k}^2+d\e_{k'}^2-2d\e_{k}\e_{k'}\bar\zeta < \k^2\s_{k,k'}^2 \big),
\end{align*}
where $\bar\zeta=\frac1d\sum_{j=1}^d \zeta_j$ with $\zeta_1,\ldots,\zeta_d$ being i.i.d.\ Rademacher random variables (\textit{i.e.}, random variables taking the values $+1$ and $-1$
with probability $1/2$). One easily checks that
$$
\frac{d\e_{k}^2+d\e_{k'}^2-\k^2\s_{k,k'}^2}{2d\e_{k}\e_{k'}} = \frac{2\s_{k}^2+2\s_{k'}^2-(\s_{k}^2+\s_{k'}^2)}{4\s_{k}\s_{k'}}\ge \frac12.
$$
Therefore, using the Hoeffding inequality, we get $\mu(\mathcal E\setminus\bar\T_\k)  \le \frac12n(n-1) \prob\big(\bar\zeta > 1/2 \big)\le \frac12n(n-1) e^{-d/8}$.
\end{proofof}


\begin{proofof}[Proof of Lemma~\ref{existence2}]
We first prove an auxiliary result.

\begin{lemma}\label{existence}
For any integer $n\ge 2$ there exist permutations $\pi_0,\pi_1,\ldots,\pi_M$ in $\Sn$ such that
$$
\pi_0=\id,\qquad\ M \ge (n/8)^{n/2}
$$
and for any pair $i,j\in\{0,\ldots,M\}$ of distinct indices we have
$\delta_H(\pi_i, \pi_j) \ge \frac{1}{2}$.
\end{lemma}

\begin{proof}
When $n\le 8$, the claim of this lemma is trivial since one can always find at least one permutation that
differs from the identity at all the positions and thus $M\ge 1\ge (n/8)^{n/2}$. Let us consider the case
$n>8$. For every $\pi\in\Sn$, denote
\begin{equation*}
E_\pi \triangleq \Big\{ \pi'\in\Sn \suchthat \frac1n \sum_{i=1}^n \fcar_{\{\pi(i)\neq \pi'(i)\}} \ge 1/2 \Big\}.
\end{equation*}
We first notice that for every $\pi\in\Sn$, there is a one-to-one correspondence between $E_{id}$ and $E_\pi$ through the bijection
\begin{equation*}
\fonction{\phi}{E_{id}}{E_\pi\phantom{\Big()}}{\pi'}{\pi\circ\pi'},
\end{equation*}
so that $\# E_\pi = \# E_{id}$. The following lemma, proved later on in this section, gives a bound for this number.
\begin{lemma}\label{cardinalEid}
Let $n\ge 2$ be an integer and $m$ be the smallest integer such that $2m\ge n$.
Then
\begin{equation*}
\# E_{id}^\complement \le \frac{4n!}{m!}.
\end{equation*}
\end{lemma}
Now we denote $\pi_0=id$ and choose $\pi_1$ in $E_{id}$. Then, it is sufficient to choose $\pi_2$ as any element from $E_{id}\cap E_{\pi_1}$,
the latter set being nonempty since
\begin{align*}
\#\big(E_{\pi_0}\cap E_{\pi_1}\big) &\ge \# \Sn - \# E_{\pi_0}^\complement - \# E_{\pi_1}^\complement \\
&\ge n!\times\Big(1-\frac{8}{m!}\Big) > 0.
\end{align*}
We can continue the construction until $\pi_i$ if
\begin{equation*}
1-\frac{4i}{m!} > 0\qquad \ssi\qquad i < \frac{m!}{4}.
\end{equation*}
To conclude, we observe that
\begin{equation*}
\frac{m!}{4} > \frac14 \Big( \frac{n}{2e} \Big)^{n/2} \ge \Big( \frac{n}{8} \Big)^{n/2}.
\end{equation*}
\vspace{-20pt}

~
\end{proof}

Let us denote by $m$ the largest integer such that $2m\le n$, and choose
$$
\tilde\pi_0,\tilde\pi_1,\ldots,\tilde\pi_M\in\mathfrak{S}_m \quad \text{ with } \quad M\ge(m/8)^{m/2}
$$
as in Lemma~\ref{existence}, so that for every $i\neq j\in\{0,\ldots,M\}$,
$\delta_H(\tilde\pi_i,\tilde\pi_j) \ge \frac{1}{2}$.
We use each permutation $\tilde\pi_i\in\mathfrak{S}_m$ to construct a permutation $\pi_i\in\Sn$.
The idea of the construction is as follows: the permutation $\pi_i$ is a product of $m$ transpositions of
distinct supports, and each transposition permutes an even integer with an odd one.
We set $\pi_0=id$ and for every $i$ in $\{1,\ldots,M\}$,
\begin{equation*}
\pi_i = \big(1 \ 2\tilde\pi_i(1)\big) \circ \big(3 \ 2\tilde\pi_i(2)\big) \circ \ldots \circ \big(2m-1 \ 2\tilde\pi_i(m)\big) \in \Sn.
\end{equation*}
With these choices, the number of differences between $\pi_i$ and $\pi_j$ is exactly twice as much as the number of differences
between $\tilde\pi_i$ and $\tilde\pi_j$. To sum up, for every pair of distinct indices $i,j\in\{0,\ldots,M\}$,
\begin{equation*}
\frac1n \sum_{k=1}^n \fcar_{\{\pi_i(k)\neq \pi_j(k)\}} = \frac{2m}{n} \times \frac1m \sum_{k=1}^m \fcar_{\{\tilde\pi_i(k)\neq \tilde\pi_j(k)\}} \ge \frac{m}{n}\ge\frac{3}{8},\qquad \forall n\ge 4.
\end{equation*}
To complete the proof, we note that $m\ge n/3$.
\end{proofof}

\begin{proofof}[Proof of Lemma~\ref{cardinalEid}]
For every $\ell\in\{m,\ldots,n\}$, counting all the permutations $\pi$ such that $\sum_{k=1}^n \fcar_{\{\pi(k)\neq k\}}=\ell$, we get
\begin{equation*}
\# E_{id} = !n + !(n-1)\binom{n}{1}  + \ldots + !(n-m)\binom{n}{m} ,
\end{equation*}
where $!\ell$ is the number of derangements, the permutations such that none of the elements appear in their original position, in $\mathfrak{S}_\ell$
for $\ell\ge 1$. We know that
\begin{equation*}
\forall \ \ell\ge1, \qquad !\ell = \ell! \times \sum_{j=0}^\ell \frac{(-1)^j}{j!},
\end{equation*}
which, using the alternating series test, yields
\begin{equation*}
\forall \ \ell\ge1, \qquad !\ell \ge \ell !\times\Big(e^{-1}-\frac{1}{(\ell+1)!}\Big).
\end{equation*}
It follows that
\begin{align*}
\# E_{id} &\ge n!\times\Big(e^{-1}-\frac{1}{(n-m+1)!}\Big)\times\Big( 1 + \frac{1}{1!} + \ldots + \frac{1}{m!} \Big) \\
&\ge n!\times\Big(e^{-1}-\frac{1}{(n-m+1)!}\Big)\times\Big(e-\frac{e}{(m+1)!}\Big)\\
&\ge n!\times\Big(1-\frac{e}{(n-m+1)!} - \frac{1}{(m+1)!} \Big).
\end{align*}
Therefore,
\begin{equation*}
\# E_{id}^\complement \le n!\times\Big( \frac{e}{(n-m+1)!} + \frac{1}{(m+1)!} \Big) \le \frac{4n!}{m!}.
\end{equation*}
\vspace{-20pt}

~
\end{proofof}

\section*{Acknowledgments}
This work was partially supported by the grants Investissements d'Avenir (ANR-11-IDEX-0003/Labex Ecodec/ANR-11-LABX-0047) and
CALLISTO. The authors thank the Reviewers for many valuable suggestions. Special thanks to an anonymous Referee for pointing a 
mistake in the proof of Theorem~\ref{lowerbound2}.

\bibliography{biblio_permut_est}

\end{document}